\documentclass{amsart}
\usepackage{hyperref}
\usepackage{amsmath}
\usepackage{amsthm}
\usepackage{cleveref}
\usepackage{amsfonts}
\usepackage{graphicx}
\usepackage{epstopdf}
\usepackage{algorithmic}
\usepackage{subcaption}
\usepackage{wrapfig}
\usepackage{soul}
\ifpdf
  \DeclareGraphicsExtensions{.eps,.pdf,.png,.jpg}
\else
  \DeclareGraphicsExtensions{.eps}
\fi

\newcommand{\e}{\mathrm{e}}

\newcommand{\He}{H^{\gamma, \varepsilon, t}}
\newcommand{\haggf}{H_{\mathbb{\lbb}}^{G, E, t}} 
\newcommand{\haggfe}{H^{G, E, t}} 

\def\eps{\varepsilon}
\renewcommand{\epsilon}{\varepsilon}

\newcommand{\bx}{\mathbf{x}}
\newcommand{\by}{\mathbf{y}}
\newcommand{\bv}{\mathbf{v}}
\newcommand{\bxi}{\ensuremath{\boldsymbol\xi}}
\newcommand{\bDelta}{\boldsymbol{\Delta}}
\newcommand{\bq}{\mathbf{q}}
\newcommand{\lbb}{\ensuremath{\boldsymbol\ell}} 

\newcommand{\dk}{\Delta_k}

\newcommand{\jmax}{{j_{\mathrm{max}}}}
\crefname{figure}{Figure}{figures}
\newtheorem{criterion}{Criterion}
\newtheorem{remark}{Remark}
\newtheorem{theorem}{Theorem}
\newtheorem{lemma}{Lemma}
\newtheorem{proposition}{Proposition}
\numberwithin{theorem}{section}
\numberwithin{lemma}{section}
\numberwithin{remark}{section}
\numberwithin{proposition}{section}
\numberwithin{equation}{section}

\title[Stable interpolation with anisotropic Gaussians]{Stable Interpolation with isotropic and anisotropic Gaussians using Hermite generating function}
\author[K. Kormann, C. Lasser, A. Yurova]{Katharina Kormann$^{\dagger\ddagger}$ \and Caroline Lasser$^{\ddagger}$ \and Anna Yurova$^{\dagger\ddagger}$}
\date{\today \\
$\phantom{111}$\textbf{Funding:} This work has been carried out within the framework of the EUROfusion Consortium and has received funding from the Euratom research and training programme 2014-2018 and 2019-2020 under grant agreement No 633053. The views and opinions expressed herein do not necessarily reflect those of the European Commission.\\
$\phantom{111}^{\dagger}$Max Planck Institute for Plasma Physics, Boltzmannstr.~2, 85748 Garching, Germany
  (\href{mailto:katharina.kormann@ipp.mpg.de}{katharina.kor\-mann@ipp.mpg.de}, \href{mailto:anna.yurova@ipp.mpg.de}{anna.yurova@ipp.mpg.de}).\\
$\phantom{111}^{\ddagger}$Technical University of Munich, Department of Mathematics, Boltzmannstr.~3, 85748 Garching, Germany (\href{mailto:classer@ma.tum.de}{classer@ma.tum.de}).} 

\begin{document}
\maketitle

\begin{abstract}
Gaussian kernels can be an efficient and accurate tool
for multivariate interpolation. In practice, high accuracies are often achieved in the flat limit where the interpolation matrix becomes increasingly ill-conditioned. Stable evaluation algorithms for isotropic Gaussians (Gaussian radial basis functions) have been proposed based on a Chebyshev expansion of the Gaussians
by Fornberg, Larsson \& Flyer and based on a Mercer expansion with Hermite polynomials by Fasshauer \& McCourt. In this paper, we propose a new stabilization algorithm for the multivariate interpolation with isotropic or anisotropic Gaussians  derived from the generating function of the Hermite polynomials. We also derive and analyse a new analytic cut-off criterion for the generating function expansion that allows to automatically adjust the number of the stabilizing basis functions. 
\end{abstract}

\section{Introduction}
Multivariate interpolation is a topic that is relevant for 
a vast number of applications.
Gaussian radial basis functions (Gaussian RBFs) are a class of functions for which interpolation generalizes to higher dimensions in a simple way while yielding spectral accuracy \cite{fasshauer2012}. However, it is known that rather small values of the \emph{shape parameter} $\eps>0$ (the width of the Gaussians) are often required. 
In this case, the
Gaussians become increasingly flat, and the interpolation matrix becomes ill-conditioned. 
This problem has been extensively studied in the literature (see the review \cite{fornberg_flyer_2015} by Fornberg \& Flyer and \cite{tarwater1985} for Tarwater's description of
this phenomenon in 1985). It has been quantified by Fornberg \& Zuev \cite{fornberg2007}, 
that the eigenvalues of the interpolation matrix are proportional to 
powers of the shape parameter, causing the notorious ill-conditioning in the flat limit regime $\eps\to0$. 

A direct collocation solution of the interpolation problem, referred to as RBF-Direct in the literature, computes the expansion coefficients of the Gaussian interpolant by inverting the collocation matrix and then evaluating the expansion. 
Several algorithms have been proposed to stabilize 
this procedure 
 in the flat limit regime, see \cite{fornberg2004, fornberg2007stable,fornberg2011stable,fasshauer2012stable,larsson2013,fasshauer2015kernel,fornberg2016}. 
A common idea of many of the stabilization algorithms---including the one proposed in this paper---is to 
evaluate the interpolant in a sequence of well-conditioned steps by a transformation to a different basis so that 
the ill-conditioning is isolated in a diagonal matrix that can be inverted analytically. 

\subsection{The new HermiteGF stabilization approach}
In this paper, we propose a stabilizing expansion of isotropic Gaussian functions, later referred to as HermiteGF expansion, built on the exponential generating function of the classic Hermite polynomials. For certain classes of functions, anisotropic Gaussians yield improved accuracy as shown in \cite{beatson2010}.
To include these cases in our description, we use an anisotropic generating function recently obtained by Dietert, Keller, \&  
Troppmann~\cite{dietert2017invariant} as well as by Hagedorn \cite{hagedorn2015generating} and generalize our HermiteGF expansion to the anisotropic case. 
We also propose and analyze a novel cut-off criterion, that contrary to the existing ones does not only account for the diagonal contributions of the stabilization but measures the impact of the full stabilization basis.

\subsection{Previous stabilization approaches}
The first stabilization method was the Contour--Pad\'e approximation proposed by Fornberg \& Wright for multi\-quadrics \cite{fornberg2004}. Later Fornberg \& Piret \cite{fornberg2007stable} proposed the so-called RBF-QR method for stable interpolation with Gaussians on the sphere by expanding the Gaussians in spherical harmonics. 
The method has been extended to more general domains in one to three dimensions by Fornberg, Larsson, \& Flyer \cite{fornberg2011stable}. This expansion is based on a combination of Chebyshev polynomials and spherical harmonics. This method will be referred to as Chebyshev-QR in this paper. The technique has also been used for the stable computation of RBF-generated finite differences by Larsson, Lehto, Heryudono, \& Fornberg~\cite{larsson2013}. Fornberg, Lehto, \& Powell~\cite{fornberg2016} developed an alternative stabilization technique for the same problem. To treat complex domains, the Chebyshev-QR method has been combined with a partition-of-unity approach by Larsson, Shcherbakov, \& Heryudono \cite{larsson2017least}. 
Fasshauer \& McCourt \cite{fasshauer2012stable} have developed another RBF-QR method, called Gauss-QR, that relies on a Mercer expansion of the Gaussian kernel. The basis transformation involves 
scaled Hermite polynomials. Compared to the Chebyshev-QR method by Fornberg et al.~\cite{fornberg2011stable}, the Gauss-QR algorithm extends to higher dimensions in a simpler way. On the other hand, the method introduces an additional parameter that needs to be hand-tuned. 
Our new basis is similar to the one in  \cite{fasshauer2012stable} with the difference that it can be extended to the interpolation with anisotropic Gaussians. Moreover, the generating function framework 
enables us to derive a new cut-off criterion that accounts for the 
full Hermite basis effect.

\subsection{Organization of the paper}
The paper is organized as follows: In \cref{sec:hermiteGFtheory}, we introduce our HermiteGF expansion of isotropic and anisotropic Gaussians and derive important properties of the HermiteGF basis. In \cref{sec:stabilization}, we propose a new RBF-QR method based on the HermiteGF basis. 
A cut-off criterion for the new HermiteGF expansion is derived in \cref{sec:criterion}. Numerical results show the accuracy of our method in \cref{sec:numericalResults} and finally, \cref{sec:conclusions} concludes the paper.

\section{HermiteGF expansion}
\label{sec:hermiteGFtheory}

In this section, similarly to the earlier stabilization approaches \cite{fornberg2007stable,fornberg2011stable,fasshauer2012stable}, we propose an expansion of the anisotropic Gaussians in a ``better'' basis, that spans the same space, but avoids instabilities related to the flat limit. For the sake of simplicity, we first derive the expansion for Gaussians in 1D. We then extend our expansion to the case of multivariate anisotropic Gaussians.

\subsection{Interpolation problem}
Before introducing our expansion of the Gaussian basis, let us briefly define the interpolation problem.
\label{sec:interpolationProblem}
Given a set $\{\phi_k(\bx)\}_{k=1}^N$ of 
expansion functions and the values $\boldsymbol{f} = \{f_i\}_{i=1}^{N}$ of the function 
$f:\mathbb{R}^d\to\mathbb{R}$
at collocation points $\{\bx_i^{\mathrm{col}}\}_{i=1}^N$ we seek to find an interpolant of the form,
\begin{equation}
 s(\bx) = \sum_{k = 1}^N \alpha_k \phi_k(\bx), 
 \label{eq:RBFinterpolation}
\end{equation}
such that it satisfies the $N$ collocation conditions,
\begin{equation*}
 s(\bx_i^{\mathrm{col}}) = f_i \quad \text{for} \quad i=1,\ldots, N.
\end{equation*}
The direct approach is to find the coefficients $\boldsymbol{\alpha}=\{\alpha_k\}_{k=1}^{N}$ as a solution of the linear system,
\begin{equation}
 \Phi^{\mathrm{col}} \boldsymbol{\alpha} = \boldsymbol{f}, \quad \text{with}\quad \Phi_{ij}^{\mathrm{col}} = \phi_j(\bx_i^{\mathrm{col}}).
 \label{eq:interpolationProblem}
\end{equation}
The matrix $\Phi^{\mathrm{col}}\in\mathbb{R}^{N\times N}$ is called \emph{collocation matrix}. Then, after solving the linear system \eqref{eq:interpolationProblem}, the interpolant \eqref{eq:RBFinterpolation} can be evaluated at any point of the domain.
In this paper, we consider Gaussian radial basis functions (isotropic Gaussians)
\begin{equation*}
\phi_k(\bx) = \exp(-\epsilon^2 \Vert \bx - \bx_k \Vert ^2)
\end{equation*}
with center points $X^{\mathrm{cen}} = \{\bx_k\}_{k=1}^N$, shape parameter $\epsilon > 0$ and anisotropic Gaussians
\[
 \phi_k(\bx) = \exp(-(\bx - \bx_k)^T E^T E (\bx - \bx_k))
\]
with invertible shape matrix $E\in\mathbb{R}^{d\times d}$.

\subsection{HermiteGF expansion in 1D}
Let $\{h_\ell\}_{\ell \geq 0}$ be the Hermite polynomials in the physicists' version, that satisfy the
recurrence relation,
\begin{equation*}
 h_{\ell+1}(x) = 2xh_\ell(x) -2\ell h_{\ell-1}(x), \quad h_{0}(x) = 1, \quad h_{-1}(x) = 0.
 \end{equation*}
The following upper bound holds for the magnitude of the Hermite polynomials \cite[Expr. 22.14.17]{abramowitz1964handbook}: There exists a positive constant $c \approx 1.086435$  
such that
\begin{equation*}
 \vert h_\ell(x) \vert \leq e^{\frac{x^2}{2}}\,c\,\sqrt{2^{\ell}\ell!}\,.
\end{equation*}
for all $\ell\ge 0$ and all $x\in\mathbb{R}$. The combinatorial factors $\sqrt{2^{\ell}\ell!}$ grow very fast with $\ell$. Therefore, in order to avoid overflow for large $\ell$, we work with a scaled version of the Hermite polynomials, $\frac{1}{\sqrt{2^{\ell} \ell!}}h_{\ell}$.
We introduce three parameters, 
\[
\varepsilon>0, \ \gamma>0, \ t \in (0,1),
\]
that is, the usual shape parameter $\eps>0$, a scaling parameter $\gamma>0$ for varying 
the evaluation domain of the Hermite polynomials to improve conditioning, and a truncation 
parameter $t\in(0,1)$ for controlling the truncation error of the stabilization expansion. We then
define the following basis functions,
\begin{equation*}
\He_\ell(x) = \frac{t^{\vert \ell \vert/2}}{\sqrt{2^\ell \ell!}}h_\ell(\gamma x)\e^{-\varepsilon^2 x^2}, 
\end{equation*}
that we refer to as \emph{HermiteGF functions}. 

Based on the generating function theory, we derive an infinite expansion of the one dimensional Gaussian RBFs in the new HermiteGF basis $\{\He_\ell\}_{\ell\ge0}$.

\begin{proposition}[HermiteGF expansion (1D)]
For all parameters $\varepsilon > 0$, $\gamma > 0$, $t\in(0,1)$, 
and for all shifts $x_0 \in \mathbb{R}$, we have a pointwise expansion
\begin{equation}
\phi_k(x) = \e^{-\varepsilon^2\left(x-x_k\right)^2} = \exp\left(\varepsilon^2 \dk ^2\left(\frac{\varepsilon^2}{\gamma^2} - 1\right)\right) \sum_{\ell \geq 0} \frac{\varepsilon^{2\ell}\sqrt{2^\ell}}{\gamma^\ell
\sqrt{t^\ell\ell!}}\dk^\ell \He_\ell(x-x_0),
\label{eq:RBFhermiteExpansion}
\end{equation}
for all $k=1,\ldots,N$, where $\dk = x_k - x_0$. 
The RBF interpolant $s(x)$ can then be pointwise computed as
\begin{equation}
 s(x) = \sum_{k = 1}^N \alpha_k\exp\left(\varepsilon^2 \dk^2\left(\frac{\varepsilon^2}{\gamma^2} - 1\right)\right)  \sum_{\ell \geq 0}  \frac{\varepsilon^{2\ell} \sqrt{2^\ell}}{\gamma^\ell \sqrt{t^\ell\ell!}} \dk^\ell \He_\ell(x - x_0).
 \label{eq:interpolantExpansion}
\end{equation}
\label{th:HermiteGF_expansion}
\end{proposition}
\begin{proof}
The Hermite polynomial's generating function is given by (see e.g.~\cite[Expr. 22.9.17]{abramowitz1964handbook}),
 \begin{equation*}
  \mathrm{e}^{2ba - a^2} = \sum_{\ell \geq 0} \frac{a^\ell}{\ell!} h_\ell(b), \quad \forall a, b \in \mathbb{R}.
 \end{equation*}
Choosing $a = \frac{\varepsilon^2 \dk}{\gamma}$ and $b = \gamma (x - x_0)$, we obtain
\begin{equation}\label{eq:hgf_expansion_proof}
 \sum_{\ell \geq 0} \frac{\varepsilon^{2\ell}}{\gamma^\ell \ell!}\dk^\ell h_\ell(\gamma (x - x_0)) = \exp\left(2\varepsilon^2\dk (x-x_0) - \frac{\varepsilon^4 \dk^2}{\gamma^2}\right)
\end{equation}
Hence, we get
\begin{align*}
&\exp\left(\varepsilon^2 \dk^2\left(\frac{\varepsilon^2}{\gamma^2}-1\right)\right) \sum_{\ell \geq 0}\frac{\varepsilon^{2\ell} \sqrt{2^\ell}}{\gamma^\ell \sqrt{t^{\ell}\ell!} }\dk^\ell\He_\ell(x - x_0)  \\
&= \exp\left(\varepsilon^2 \dk^2\left(\frac{\varepsilon^2}{\gamma^2}-1\right)\right) \sum_{\ell \geq 0}\frac{\varepsilon^{2\ell} }{\gamma^\ell \ell!}\dk^\ell h_\ell(\gamma (x - x_0))\mathrm{e}^{-\varepsilon^2(x - x_0)^2} 
 \stackrel{\eqref{eq:hgf_expansion_proof}}{=} \mathrm{e}^{-\varepsilon^2(x - x_k)^2},
\end{align*}
which proves expansion \eqref{eq:RBFhermiteExpansion}. Using expansion \eqref{eq:RBFhermiteExpansion} in the interpolant \eqref{eq:RBFinterpolation}, we get the representation \eqref{eq:interpolantExpansion}.
\end{proof} 

\begin{remark}[Basis centering]
Hermite polynomials are symmetric with respect to the axis $x=0$. Due to its growth behavior, 
it is advantageous to have the basis centered around this point of symmetry, that is, to use the translation $x_0 = \frac{B-A}{2}$, where $[A,B]$ is the interval of interest for evaluating the function $f$.
\end{remark}

\begin{remark}[The parameter $\gamma$]
 The parameter $\gamma>0$ in the basis $\{\He_{\ell}\}_{\ell\ge0}$ allows control over the evaluation domain of the Hermite polynomials. When choosing it, one has to consider two counteracting effects: For small values of 
 $\gamma$, 
ill-conditioning appears since the values of the basis functions at the collocation points are too similar. On the other hand, Hermite polynomials take very large values on large domains which can lead to an overflow. An optimal balance depends on the particular function and the number of basis functions.
 \end{remark}
 
\subsection{Multivariate HermiteGF expansion of anisotropic Gaussians}
The HermiteGF expansion can be easily extended to higher dimensions for the case of isotropic Gaussians, using tensor products of 1D physicists' Hermite polynomials,
\begin{equation*}
h_{\lbb}(\bx) = h_{\ell_1}(x_1) \cdots \cdot h_{\ell_d}(x_d),
\quad \lbb\in\mathbb{N}^d,\quad \bx\in\mathbb{R}^d.
\end{equation*}
However, finding a stable interpolant for anisotropic Gaussian functions of the type
\[
\phi_{\mathbf{q}}(\bx) = \exp(-(\bx-\bq)^TE^TE(\bx-\bq)),\qquad \bx,\mathbf{q}\in\mathbb{R}^d,
\]
is a more challenging task. A similar question was raised in \cite[$\mathsection$ 8.5]{fasshauer2012stable}, however, without further investigation. McCourt \& Fasshauer \cite{mccourt2017stable} considered anisotropic Gaussians with diagonal shape matrix $E$ using Mercer expansion theory, but this result has not been extended to the case of arbitrary $E$. 
Analogously to the 1D case, we define the multivariate version of our \emph{HermiteGF functions} by
\begin{equation*}
 H_{ \lbb }^{G, E,t}(\bx) = \frac{t^{\vert \lbb \vert/2}}{\sqrt{2^{\vert \lbb \vert} \lbb!}}h_{\lbb}(G^T \bx) \exp(-\bx^T E^T E \bx),
\end{equation*}
where $G, E \in \mathbb{R}^{d\times d}$ are arbitrary invertible matrices. 


\begin{proposition}[HermiteGF expansion of anisotropic Gaussians]
Let $\bq \in \mathbb{R}^d$ and $E, G \in \mathbb{R}^{d \times d}$ invertible matrices.
Consider the anisotropic Gaussian $\phi_{\bq}(\bx)$. Then, for any 
shift $\bx_0 \in \mathbb{R}^d$ the following relation holds 
pointwise in $\bx\in\mathbb{R}^d$:
 \begin{equation}
  \phi_{\bq}(\bx) = \exp(\bDelta_{\bq}^T (\tilde{G}-E^TE)\bDelta_{\bq})
  \cdot \sum_{\lbb \in \mathbb{N}^d} \frac{(G^{-1}E^T E\bDelta_{\bq})^{\lbb} \sqrt{2^{\vert \lbb \vert}}}{\sqrt{t^{\vert \lbb \vert}\lbb!}} \haggf(\bx - \bx_0),  
  \label{eq:RBFhermiteExpansionAnisotropic}
 \end{equation}
 where $\bDelta_{\bq} = \bq - \bx_0$, $\tilde{G} = E^T E G^{-T} G^{-1} E^T E$.
\end{proposition}

\begin{proof}
We use the following anisotropic Hermite generating function (see 
\cite[Lemma 5]{dietert2017invariant} or \cite[Theorem 3.1]{hagedorn2015generating} with 
$A = \mathrm{Id}_d$): 
\begin{equation}
  \sum_{\lbb \in \mathbb{N}^d} \frac{\mathbf{a}^{\lbb}}{\lbb!} h_{\lbb}(\mathbf{b}) = \exp(2\mathbf{b}^T\mathbf{a} - \mathbf{a}^T\mathbf{a}),\qquad\forall \mathbf{a}, \mathbf{b} \in \mathbb{R}^d.
  \label{eq:hermiteGenFuncAnisotropic}
  \end{equation}
We denote $\mathbf{b} = G^T(\bx - \bx_0)$ and $\mathbf{a} = G^{-1}E^T E\bDelta_{\bq}$. Then, using \eqref{eq:hermiteGenFuncAnisotropic}, we get
\begin{equation*}
 \sum_{\lbb\in\mathbb{N}^d}\frac{(G^{-1}E^T E\bDelta_{\bq})^{\lbb}}{\lbb!} h_{\lbb}(G^T (\bx - \bx_0)) = \exp(2(\bx - \bx_0)^TE^T E\bDelta_{\bq} - \bDelta_{\bq}^T\tilde{G}\bDelta_{\bq}).
\end{equation*}
We observe that
\begin{align*}
    2\bx^TE^TE\bq &= 2(\bx - \bx_0)^TE^TE\bDelta_{\bq} + 2\bx^TE^TE\bx_0 + 2\bx_0^TE^TE\bDelta_{\bq},
 \\
  -\bx^T E^T E \bx &= -(\bx - \bx_0)^TE^TE (\bx - \bx_0) - 2\bx^T E^TE\bx_0 + \bx_0^TE^TE\bx_0.
\end{align*}
Putting everything together we arrive to \eqref{eq:RBFhermiteExpansionAnisotropic}.
\end{proof}
This expansion provides a new powerful tool for dealing with anisotropic approximation. Note that the standard multidimensional isotropic Gaussian interpolation corresponds to the following matrix $E$:
\begin{equation*}
 E_{\mathrm{isotropic}} = \eps \mathrm{Id}_d.
\end{equation*}

\subsection{Mehler's formula (nD)}
\label{sec:mehler_nD}
In this section, we investigate the behavior of the HermiteGF basis functions with large index $\lbb$ approaching infinity. To be able to decide where to cut the HermiteGF expansion for numerical purposes, it is useful to quantify the size of the tail of the truncated expansion. 
For that, we take a look at the magnitude of the values of the HermiteGF basis. We consider the infinite-dimensional vector that contains the values of all basis functions $\haggf(\bx)$ at a certain point $\bx \in \mathbb{R}^d$. Its Euclidean norm can be computed analytically using a multivariate extension of Mehler's formula for Hermite polynomials.
\begin{theorem}[Bilinear generating function]
 \label{th:mehlerMultidim}
 The following relation holds for all $t \in (0, 1)$ and all $\bx, \by \in \mathbb{R}^d$
 \begin{equation}
  \sum_{\vert \lbb \vert = 0}^{\infty} \frac{t^{\vert \lbb \vert} h_{\lbb}(\bx)h_{\lbb}(\by)}{2^{\vert \lbb \vert} \lbb!} = \frac{\exp\left(\frac{t}{1-t^2}(\bx^T \by + \by^T \bx) - \frac{t^2}{1-t^2}(\Vert \bx \Vert_2^2 + \Vert \by \Vert_2^2)\right)}{(1 - t^2)^{d/2}}.
  \label{eq:mehlerMultidim}
 \end{equation}
\end{theorem}
\begin{proof}
 We extend the 1D proof proposed by Watson in \cite[p. 4]{watson1933notes} to multiple dimensions. Applying the inverse Fourier transform to the Fourier transform of a normal distribution with $\sigma = 1/\sqrt{2}\mathrm{Id}_d$, we get
\begin{equation*}
\e^{\bx^T\bx} = \pi^{d/2}\int_{\mathbb{R}^d}\e^{2\pi \mathrm{i}\bx^T\bxi} \e^{-\pi^2\bxi^T\bxi} \mathrm{d}\bxi = \pi^{-d/2}\int_{\mathbb{R}^d}\e^{2 \mathrm{i}\bx^T\bxi} \e^{-\bxi^T\bxi} \mathrm{d}\bxi.
\end{equation*}
Hence, 
\[
 h_{\lbb}(\bx) = \e^{\bx^T\bx}(-\nabla)^{\lbb} \e^{-\bx^T\bx} = \frac{(-2\mathrm{i})^{\vert \lbb \vert}}{\pi^{d/2}}\e^{\bx^T \bx} \int_{\mathbb{R}^d}\bxi^{\lbb} \e^{2\mathrm{i}\bx^T\bxi - \bxi^T \bxi} \mathrm{d}\bxi,
\]
where we used the Rodrigues formula for multivariate tensor-product Hermite polynomials (see \cite[Expr. 11]{dietert2017invariant} with $M = \mathrm{Id}$). Then, 
\begin{align}
&
\sum_{\vert \lbb \vert = 0}^{\infty} \frac{t^{\vert \lbb \vert} h_{\lbb}(\bx)h_{\lbb}(\by)}{2^{\vert \lbb \vert} \lbb!} 
=\nonumber\\
&=\frac{\e^{\bx^T\bx + \by^T\by}}{\pi^d}\sum_{\vert \lbb \vert = 0}^{\infty} \frac{(-2t)^{\vert \lbb \vert}}{\lbb!} \int_{\mathbb{R}^d}\int_{\mathbb{R}^d} \bxi_{\bx}^{\lbb} \bxi_{\by}^{\lbb}\e^{2\mathrm{i}(\bx^T\bxi_{\bx} + \by^T\bxi_{\by})} \e^{-\bxi_{\bx}^T\bxi_{\bx} - \bxi^T_{\by}\bxi_{\by}}  \mathrm{d}\bxi_{\bx} \mathrm{d}\bxi_{\by} \nonumber \\
&= \frac{\e^{\bx^T\bx + \by^T\by}}{\pi^d} \int_{\mathbb{R}^d}\int_{\mathbb{R}^d} \e^{-2t\bxi^T_{\bx}\bxi_{\by}}\e^{2\mathrm{i}(\bx^T\bxi_{\bx} + \by^T\bxi_{\by})} \e^{-\bxi_{\bx}^T\bxi_{\bx} - \bxi^T_{\by}\bxi_{\by}}  \mathrm{d}\bxi_{\bx} \label{eq:mehlerIntegral} \mathrm{d}\bxi_{\by},
\end{align}
where we used the Taylor series of the exponential function. Recall the following formula for 
a bivariate Gaussian Fourier integral (see \cite[p. 4]{watson1933notes}):
\begin{equation}
 \int_{-\infty}^{\infty} \int_{-\infty}^{\infty} g_{(x, y)}(u,v) \mathrm{d}v \mathrm{d}u = 
 \frac{\pi\e^{-(x^2 + y^2)/2}}{\sqrt{1-t^2}}\exp\!\left(\frac{x^2 - y^2}{2} - \frac{(x-yt)^2}{1-t^2}\right).
 \label{eq:watsonFormula}
\end{equation}
with
\[
 g_{(x, y)}(u,v) = \exp(-u^2-2tuv - v^2 + 2\mathrm{i}xu + 2\mathrm{i}yv).
\]
\noindent
Using \eqref{eq:watsonFormula} $d$ times for the integral \eqref{eq:mehlerIntegral} 
together with the algebraic identity
\[
-\tfrac12(x^2+y^2) + \tfrac12(x^2-y^2) - \frac{(x-yt)^2}{1-t^2} = 
\frac{2txy-x^2-y^2}{1-t^2}
\]
yields \eqref{eq:mehlerMultidim}.

\end{proof}
With the help of the multidimensional Mehler's formula we can now compute the 
square of the Euclidean norm 
\[
\haggfe_{\mathrm{lim}}(\bx) := 
\sum_{\vert\lbb\vert = 0}^\infty \haggf(\bx)^2 = 
\sum_{\vert \lbb \vert=0}^{\infty} \frac{t^{\vert \lbb \vert}}{2^{|\lbb|}\lbb!}h^2_{\lbb}(G^T\bx)\exp(-2\bx^TE^TE\bx)
\]
of the infinite vector containing the values of all basis functions $\haggf(\bx)$ with $\lbb \in \mathbb{N}^d$ at some point $\bx\in\mathbb{R}^d$. Indeed,
\begin{equation}
 \haggfe_{\mathrm{lim}}(\bx) = \frac{\exp\left(-2\bx^TE^TE\bx + \frac{2t\Vert G^T\bx\Vert_2^2}{1+t}\right)}{(1 - t^2)^{d/2}}.
 \label{eq:HinftyNorm}
\end{equation}

We take a look if this value matches our numerical result. One can see in \cref{fig:mehlerLimit} that for the square domain the limit value is matched for different values of $\eps$, $t$. 

\begin{figure}[h]
 \centering
\begin{subfigure}[t]{0.42\textwidth} 
\includegraphics[scale=0.33]{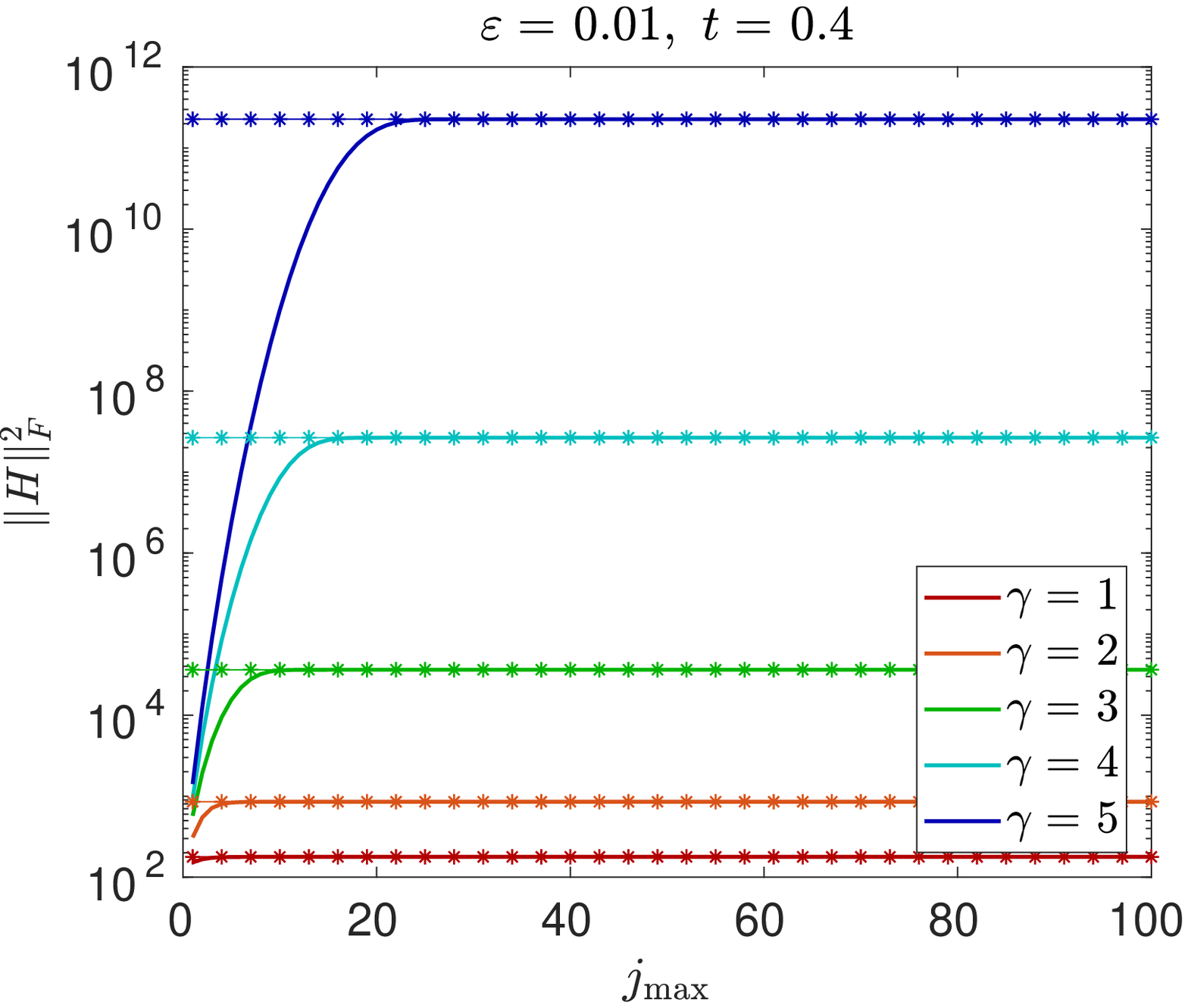} 
 \caption{$\eps = 0.01$, $t = 0.4$.}
\end{subfigure} 
\begin{subfigure}[t]{0.45\textwidth} 
\includegraphics[scale=0.33]{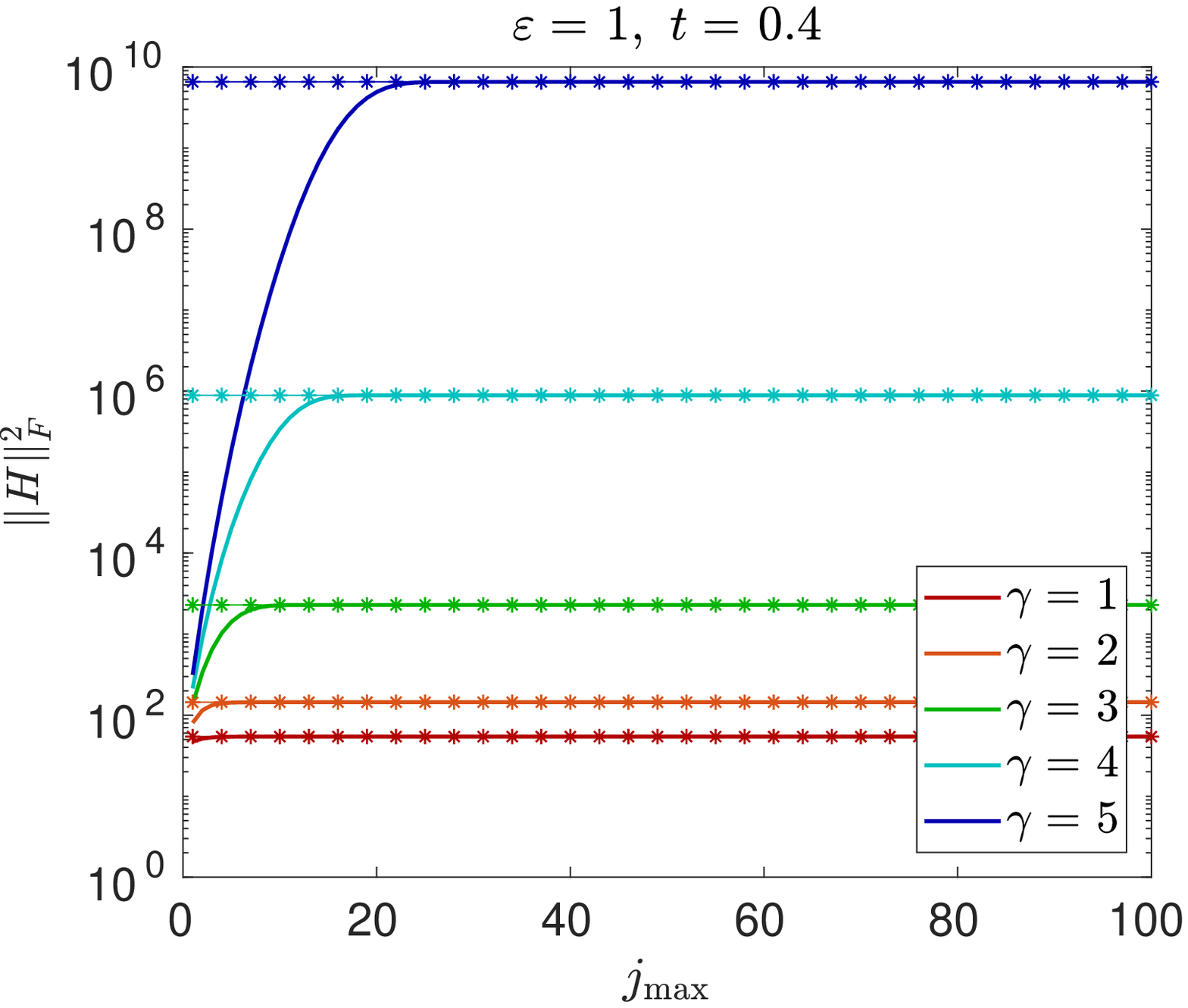}
 \caption{$\eps = 1$, $t = 0.4$.}
\end{subfigure} 
\caption{Frobenius norm of the matrix $H$ of values of the basis functions at $100$ Halton nodes $\bx_k$ on a square domain $[-1, 1] \times [-1, 1]$. The solid line is the value corresponding to only the basis functions with a degree up to $\jmax$. The \texttt{*} corresponds to the analytically computed value.} 
\label{fig:mehlerLimit}
\end{figure}
\section{Stabilization of the RBF interpolation}\label{sec:stabilization}
In this section, we derive a numerical stabilization algorithm for RBF interpolation based on the HermiteGF expansion. The main idea is to represent the RBF interpolant in the more stable basis $\{\haggfe_{\lbb}\}_{\lbb\geq 0}$. For appropriately chosen parameters $G$ and $t$, we expect it to be better conditioned. We now first derive the HermiteGF-QR algorithm in 1D and then generalize it to a multidimensional form.

\subsection{HermiteGF-QR (1D)}\label{subsec:QR1D}
In 1D we denote by
\[
\Phi(x, X^{\mathrm{cen}}) = \begin{pmatrix}
  \phi_1(x),&\ldots,&\phi_N(x)
 \end{pmatrix}
 \]
 the vector of Gaussians with center points $X^{\mathrm{cen}}$ evaluated at $x\in\mathbb{R}$ and write the stabilization expansion \eqref{eq:RBFhermiteExpansion}
as an infinite matrix-vector product
\begin{equation}
\Phi(x, X^{\mathrm{cen}})  =
\He(x-x_0)\,
 B(\eps, \gamma, t, X^{\mathrm{cen}})
 \label{eq:expansionMatrixForm}
\end{equation}
where the vector 
\[
\He(x-x_0) = \big(
  \He_0(x - x_0), \,
  \He_1(x - x_0), \,
  \ldots
 \big)
\]
contains all the elements of the polynomial basis $\{\He_\ell\}_{\ell\ge0}$ evaluated in the translated point $x-x_0$ and 
\begin{equation*}
 B(\eps, \gamma, t, X^{\mathrm{cen}})_{\ell k} = \exp\left(\varepsilon^2 \Delta_k^2\left(\frac{\varepsilon^2}{\gamma^2} - 1\right)\right) \frac{\varepsilon^{2\ell}\sqrt{2^{\ell}}}{\gamma^{\ell} \sqrt{t^\ell\ell!}}\Delta_k^{\ell}
\end{equation*}
is an $\infty\times N$ matrix.
The major part of the ill-conditioning is now confined in the matrix $B$. Since $B$ is independent of the point $x$ where the basis function is evaluated, both the evaluation and interpolation matrix can be expressed in the form \eqref{eq:expansionMatrixForm} with the same matrix $B$. 

We follow the RBF-QR approach and further split 
\[
B^T = CD
\] 
into a well-conditioned full $N\times\infty$ matrix $C$ and an infinite diagonal matrix $D$, where all harmful effects are confined in $D$. In the case of expansion~\eqref{eq:RBFhermiteExpansion}, the following setup follows naturally from the Chebyshev-QR theory \cite[$\mathsection$ 4.1.3]{fornberg2011stable},
\begin{align*}
C_{k\ell} = \exp\left(\epsilon^2 \Delta_k^2\left(\frac{\epsilon^2}{\gamma^2} - 1\right)\right) \frac{\Delta_k^{\ell}}{L^{\ell}}, \quad D_{\ell \ell} = \frac{\epsilon^{2\ell}\sqrt{2^{\ell}}}{\gamma^{\ell} \sqrt{t^\ell\ell!}} \,L^{\ell}.
\end{align*}
Here we also divide each coefficient by the radius of the domain $L$ containing the center points in order to avoid ill-conditioning in $C$ coming from taking high powers of $x_k$. That might be dangerous when the domain is too large, however, it still extends the range of domain diameters possible. 

The goal is now to find a basis $\{\psi_j\}$ spanning the same space as $\{\phi_k\}$ but yielding a better conditioned collocation matrix. In particular, we need an invertible matrix~$X$ such that $X^{-1}\Phi(x)^T$ is better conditioned. 
Let us perform a QR-decomposition on $C = QR$. Then we get,
\begin{equation*}
 \Phi(x)^T = CD\He(x-x_0)^T = Q
 \begin{pmatrix}
 R_1 & R_2                     
 \end{pmatrix}
 \begin{pmatrix}
  D_1 & 0 \\
  0 & D_2
 \end{pmatrix}\He(x - x_0)^T
\end{equation*}
where $R_1$ and $D_1$ are $N\times N$ matrices containing the upper (left) block 
of the infinite matrices $R$ and $D$, respectively, while $R_2$ and $D_2$ assemble the remaining entries. 
Consider $X = QR_1D_1$. The new basis $\Psi(\bx)^T := X^{-1} \Phi(x)^T$ can be formed as,
\begin{align*}
\Psi(x)^T & = D_1^{-1} R_1^{-1} Q^{\mathrm{H}}\Phi(x)^T \\
&=  \ D_1^{-1} R_1^{-1} Q^{\mathrm{H}} Q
 \begin{pmatrix}
  R_1 D_1 &
  R_2 D_2
 \end{pmatrix}\He(x - x_0)^T\\
&= \begin{pmatrix}
   \mathrm{Id} &
   D_1^{-1}R_1^{-1}R_2 D_2
  \end{pmatrix}
\He(x - x_0)^T.
\end{align*}
To avoid under/overflow in the computation of $D_1^{-1}R_1^{-1}R_2 D_2$, we form the two matrices $\tilde R = R_1^{-1}R_2$ and $\tilde D$ with elements 
\begin{equation*}
\tilde D_{i,j} = \left(\frac{\varepsilon^2 L}{\gamma} \sqrt{\frac{2}{t}} \right)^{N+j-i}
\sqrt{\frac{i!}{(N+j)!} },
\qquad 1\le i\le N, \quad j\ge 1,  
\end{equation*}
and compute their Hadamard product, $\tilde R .* \tilde D$.  
Despite the harmful effects contained in $D$, the resulting term $D_1^{-1} R_1^{-1}R_2 D_2 = \tilde R .* \tilde D$ is then harmless.
\subsection{Multivariate anisotropic HermiteGF-QR}
In this section, we derive an analog of the HermiteGF-QR algorithm for the multivariate case. Since we are now dealing with matrices, in the general case, it is impossible to separate~$E$ from $\bx_k$ in $(E\bx_k)^{\lbb}$ as we did before. Therefore, the flow of the HermiteGF-QR does not apply directly. However, it is still possible to tackle some part of the ill-conditioning analytically. 
Consider the following splitting of the matrix $G^{-1}E^TE \in \mathbb{R}^{d \times d}$:
\begin{equation*}
 G^{-1}E^TE = \mathrm{Diag} + \mathrm{Rem},
\end{equation*}
where $\mathrm{Diag}$ is the $d \times d$ diagonal matrix containing the diagonal elements of $G^{-1}E^T E$ and $\mathrm{Rem}$ contains the remaining off-diagonal terms. Denote 
\begin{equation*}
 \bv_k = (\mathrm{Id} + \mathrm{Diag}^{-1} \mathrm{Rem})\bDelta_k \quad \text{and} \quad \boldsymbol{\Delta}_k = \bx_k - \bx_0.
\end{equation*}
Then, it holds that
\begin{align*}
(G^{-1}E^T E\bDelta_k)^{\lbb} &= ((\mathrm{Diag} + \mathrm{Rem})\bDelta_k)^{\lbb} = \left(\mathrm{Diag}\left((\mathrm{Id} + \mathrm{Diag}^{-1} \mathrm{Rem})(\bx_k - \bx_0) \right)\right)^{\lbb} \\
&= \prod_{i=1}^d (\mathrm{Diag}_{ii} (\bv_k)_i)^{\ell_i} = \prod_{i=1}^d \mathrm{Diag}_{ii}^{\ell_i} (\bv_k)_i^{\ell_i} = \left(\prod_{i=1}^d \mathrm{Diag}_{ii}^{\ell_i} \right) \bv_k^{\lbb},
\end{align*}
where $\mathrm{Diag}^{-1} \mathrm{Rem}$ can be computed analytically. Denote $\mathbf{d_\mathrm{vec}} = \mathrm{diag}(\mathrm{Diag})$. 
Then, the generating function expansion \eqref{eq:RBFhermiteExpansionAnisotropic} can be written as:
\begin{align*}
\phi_k(\bx) &=\exp(-\bDelta_k^T E^T E \bDelta_k + \bDelta_k^T \tilde{G}\bDelta_k) \sum_{\lbb \in \mathbb{N}^d} \frac{(G^{-1}E^T E\bDelta_k)^{\lbb} \sqrt{2^{\vert \lbb \vert}}}{\sqrt{t^{\vert \lbb \vert}\lbb!}} \haggf(\bx - \bx_0) \\
 &= \exp(-\bDelta_k^T E^T E \bDelta_k + \bDelta_k^T \tilde{G}\bDelta_k) \sum_{\lbb \in \mathbb{N}^d} \frac{\mathbf{d}_\mathrm{vec}^{\lbb} \bv_k^{\lbb}\sqrt{2^{\vert \lbb \vert}}}{\sqrt{t^{\vert \lbb \vert}\lbb!}} \haggf(\bx - \bx_0).
\end{align*}
As before, we can write the expansion above as the infinite matrix-vector product
\begin{equation*}
 \Phi(\bx) =
  \haggfe(\bx - \bx_0)
 \, B(E, G, t, X^{\mathrm{cen}})
\end{equation*}
with
\begin{align*}
 B(E, G, t, X^{\mathrm{cen}})_{\lbb k} &= \exp(-\bDelta_k^T E^T E \bDelta_k + \bDelta_k^T \tilde{G} \bDelta_k) \frac{\mathbf{d}_\mathrm{vec}^{\lbb} \bv_k^{\lbb}\sqrt{2^{\vert \lbb \vert}}}{\sqrt{t^{\vert \lbb \vert}\lbb!}}.
\end{align*}  
As before, we write the transpose of the infinite matrix $B$ as a product $CD$ with  
\begin{equation}
 C_{k\lbb} = \exp(-\bDelta_k^T E^T E \bDelta_k + \bDelta_k^T \tilde{G}\bDelta_k) \bv_k^{\lbb}, \quad
 D_{\lbb \lbb} = \frac{\mathbf{d}_\mathrm{vec}^{\lbb} \sqrt{2^{\vert \lbb \vert}}}{\sqrt{t^{\vert \lbb \vert}\lbb!}}.
 \label{eq:CDanisotropic}
\end{equation}
The $d\times d$ matrix product $\mathrm{Diag}^{-1} \mathrm{Rem}$ contained in the vectors $\bv_k$ can be computed analytically. The diagonal part of the matrix $G^{-1}E^TE$, that is now in the matrix $D$ can be  handled in the exact same fashion as it has been done in 1D. In particular, we perform the QR-decomposition of the matrix $C=QR$, block decompose
\[
R = \begin{pmatrix}R_1 & R_2\end{pmatrix},\qquad D = \begin{pmatrix}D_1 & 0\\ 0 & D_2\end{pmatrix},
\]
such that the entries related to the first $N$ stabilizing basis functions are contained in the $N\times N$ matrices 
$R_1$ and $D_1$, and consider the preconditioner $X = QR_1D_1$. Hence, analogously to the HermiteGF-QR case, the new basis  can be formed as
\begin{equation*}
\Psi(\bx)^T := X^{-1} \Phi(\bx)^T = \begin{pmatrix}
   \mathrm{Id} &
   D_1^{-1}R_1^{-1}R_2 D_2
  \end{pmatrix}
\haggfe(\bx-\bx_0)^T.
\end{equation*}
The action of $D_1^{-1}$ and $D_2$ can again be computed as the Hadamard product with
\begin{equation}
 \tilde{D}_{ij} = \frac{D_{j+N, j+N}}{D_{ii}} \quad \text{with} \quad i \in \{1, \ldots, N\}, \,\, 
 j \ge 1.
 \label{eq:DtildeAnisotropic}
\end{equation}

\begin{remark}
 When the magnitude of $\bv_k$ gets too large, the matrix $C$ can also become ill-conditioned. To avoid that, one can increase the magnitude of the elements of $G$. Alternatively, one can add a scaling in $C$ which should then be compensated in the matrix $D$, similarly to the scaling with the domain size $L$ in the 1D version.
\end{remark}

\subsubsection{Isotropic HermiteGF-QR}
In the isotropic case, when $E = \epsilon \mathrm{Id}_d$ and $G = \gamma \mathrm{Id}_d$, the expressions for the matrices $C$ and $D$ can be written in a simpler way. In particular, in this case, we have
\[
 \mathrm{Diag} = \gamma^{-1}\epsilon^2 \mathrm{Id}_d, \quad \mathrm{Rem} = \mathbf{0}_d, \quad \text{and} \quad \bv_k = \bx_k - \bx_0 = \boldsymbol{\Delta}_k.
\]
Hence, the diagonal vector simplifies to $\mathbf{d_\mathrm{vec}}=\gamma^{-1}\epsilon^2(1,\ldots,1)$
and the elements of the matrices $C$ and $D$ take the form
\[
 C_{k\ell} = \exp\left(\eps^2\left(\frac{\eps^2}{\gamma^2} - 1\right) \Vert \bDelta_k \Vert_{\ell_2}^2\right)\bDelta_k^{\lbb}, \quad D_{\ell \ell} = \frac{(\sqrt{2}\gamma^{-1}\eps^2)^{\vert \lbb \vert}}{\sqrt{t^{\vert \lbb \vert} \lbb!}}.
\]
\subsection{HermiteGF interpolant} In the new basis, we can write the equivalent formulation of the interpolant \eqref{eq:RBFinterpolation} as
\[
 s(\bx) = \Psi(\bx)(\Psi^{\mathrm{col}})^{-1}\boldsymbol{f} \quad \text{with} \quad \Psi_{ij}^{\mathrm{col}} = \psi_j(\bx_i^{\mathrm{col}}).
\]
To compute the interpolant in the new formulation numerically, we have to cut the infinite expansion \eqref{eq:RBFhermiteExpansionAnisotropic}. We discuss the cut-off strategy in \cref{sec:criterion}.
\subsection{Alternative splitting based on the Vandemonde matrix}
\label{sec:vandermonde}
Instead of using a QR-decomposition of the matrix $C\in\mathbb{R}^{N\times\infty}$ one could also split it  as $C = \bar{E} W$, where $\bar{E} \in \mathbb{R}^{N \times N}$ is a diagonal matrix for the exponential part and $W \in \mathbb{R}^{N\times \infty}$ accounts for the polynomial contributions,
\begin{equation}
 \bar{E}_{kk} = \exp(-\bDelta_k^T E^T E \bDelta_k + \bDelta_k^T \tilde{G}\bDelta_k) \quad \text{and} \quad
 W_{k\ell} = \bv_k^{\lbb}.
 \label{eq:EW_vandermondeSplitting}
\end{equation}
We now decompose the original basis using this splitting,
\begin{align*}
 \Phi(\bx)^T &= C \begin{pmatrix}
                            D_1 & 0 \\
                            0 & D_2
                            \end{pmatrix}\haggfe(\bx - \bx_0)^T\\
 &= \bar{E} \begin{pmatrix}
                             W_1 & W_2
                            \end{pmatrix} \begin{pmatrix}
                            D_1 & 0 \\
                            0 & D_2
                            \end{pmatrix}\haggfe(\bx - \bx_0)^T \\*[1ex]
                            &= \bar{E}\begin{pmatrix}
                            W_1 D_1 & W_2 D_2
                            \end{pmatrix}\haggfe(\bx-\bx_0)^T,
\end{align*}
where $W_1 \in \mathbb{R}^{N \times N}$ and $W_2 \in \mathbb{R}^{N \times \infty}$. With the preconditioner $X_V = \bar{E} W_1 D_1$, the new basis reads as
\begin{equation*}
 \Psi_{V}(\bx)^T = \begin{pmatrix}
                                   \mathrm{Id} & D_1^{-1} W_1^{-1} W_2 D_2
                                   \end{pmatrix}\haggfe(\bx - \bx_0)^T.
\end{equation*}
One can see that
\[
 X_V = \bar{E}W_1D_1 = C_1 D_1 = QR_1 D_1= X,
\]
where $C_1$ is the first $N \times N$ block of $C$. Therefore, in exact arithmetic $\Psi$ and $\Psi_V$ are the same, however, in floating-point arithmetic the values of the bases might differ. However, we will use this alternative splitting for the derivation and analysis of a suitable cut-off criterion for the HermiteGF basis in the next section.
\section{Cut-off of the expansion}
\label{sec:criterion}

To make the RBF-QR methods usable for numerical computations, one has to cut the expansion \eqref{eq:RBFhermiteExpansionAnisotropic} at a certain polynomial degree $\jmax\in\mathbb{N}$ which in 1D also corresponds to the number of stabilizing basis functions $M$. In the multivariate setting, the 
number of basis functions $M$ equals
\[M=\begin{pmatrix}\jmax + d \\d\end{pmatrix}.\] 
However, choosing an efficient cut-off degree $\jmax$ is not a trivial task. We first derive a criterion that is analogous to the state-of-the-art criteria for other RBF-QR methods (\cite[Expr. 5.2]{fornberg2011stable}, \cite[Expr. 4.10]{fasshauer2012stable}). However, especially in the HermiteGF case, it turns out to be inefficient, i.e.\ it overestimates the number $\jmax$. We then derive a new cut-off criterion based on the theoretical framework presented in the previous sections. This new criterion allows to directly control the approximation error of the stable basis which is more efficient while still being effective.

\subsection{State-of-the-art criterion}
Similarly to \cite[Expr. 5.2]{fornberg2011stable}, \cite[Expr. 4.10]{fasshauer2012stable} we take a look at the matrix $\tilde{D}$ that contains the effects of $D_1^{-1}$ and $D_2$. Recall that for the QR methods in \cref{subsec:QR1D} the matrix $\tilde{D}$ is then multiplied element-wise with the matrix $\tilde{R}$. We stop once all elements of the new block of $\tilde{D}$ are below machine precision, i.e.\,
\begin{equation}
 \max_{i=1\ldots N, \vert \mathbf{j} \vert \geq \jmax+1} \tilde{D}_{ij}< \eps_{\mathrm{mach}}.
 \label{eq:jmaxOld}
\end{equation}
The criterion \eqref{eq:jmaxOld} guarantees that all additional columns that could be added to $D_2$ would yield elements in $\tilde{D}$ that are below machine precision.
%
We now take a look at the behavior of the elements of the matrix $D$. Here, for the Gauss-QR method we used $\alpha_{\mathrm{Gauss-QR}} = \gamma$. One can see in \cref{fig:Doscillations} that for small values of $\eps$ the behavior is very similar for all three methods. However, for large $\eps$ the decay in~$D$ is particularly bad in our formulation.
 This criterion also neglects the matrix~$\tilde{R}$ and the effect of the polynomial vector 
 $H^{\gamma,\varepsilon,t}(x-x_0)$. In particular, we know from \cref{sec:mehler_nD} that the tail of the polynomial vector has some decay.

\begin{figure}[t]
 \centering
\begin{subfigure}[t]{0.42\textwidth} 
\includegraphics[scale=0.33]{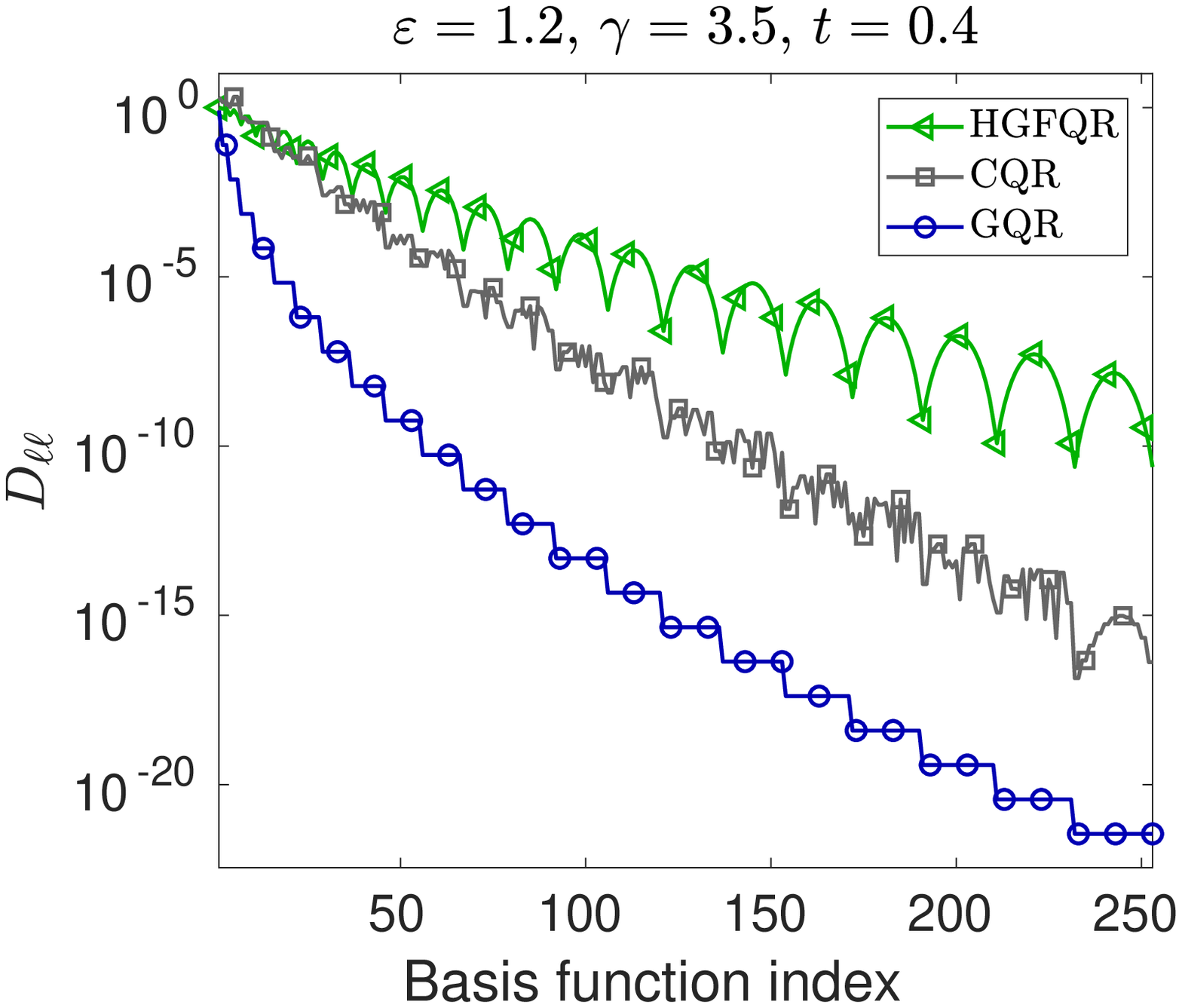} 
 \caption{Larger $\eps = 1.2$. }
\end{subfigure} 
\begin{subfigure}[t]{0.45\textwidth} 
\includegraphics[scale=0.33]{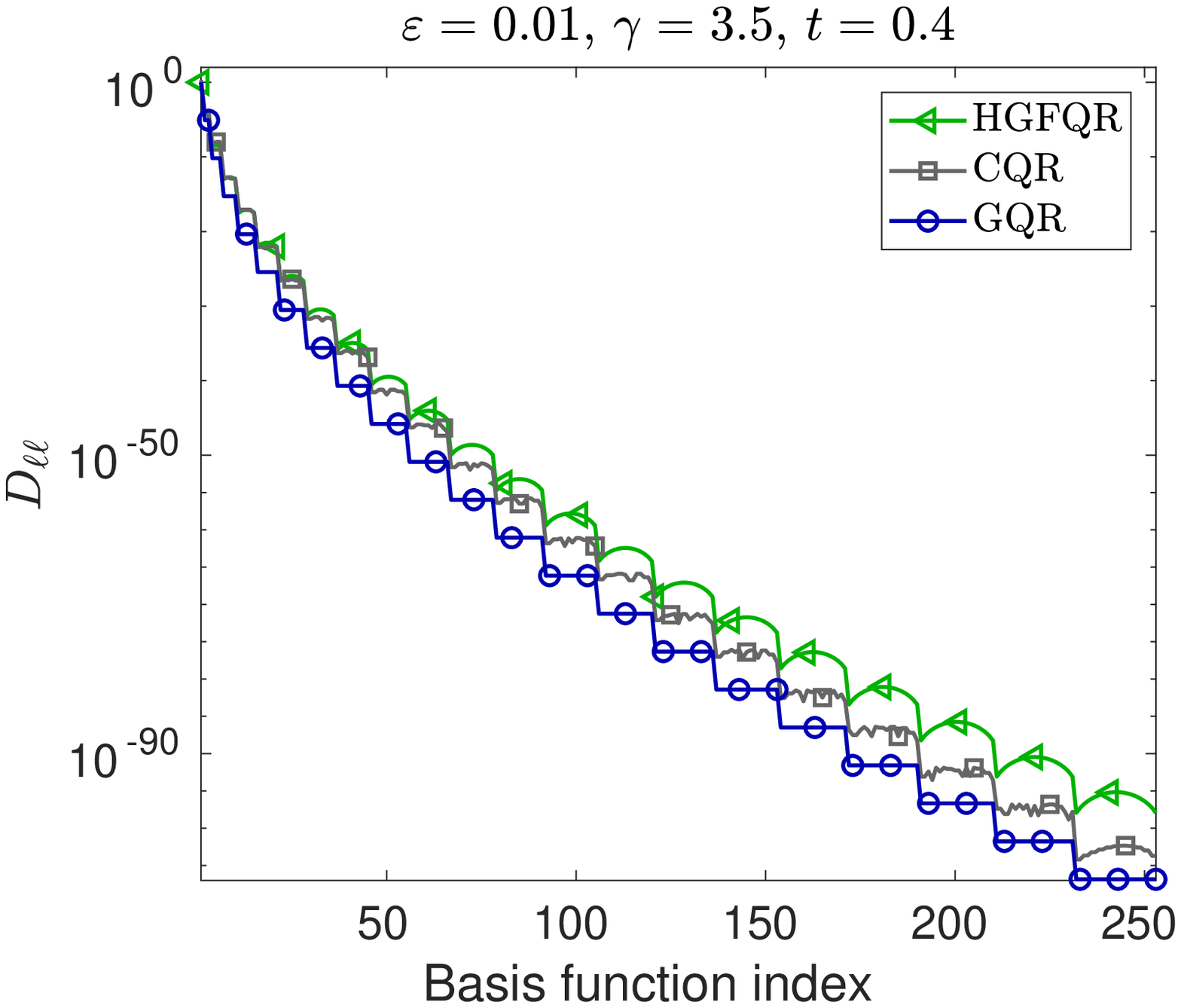}
 \caption{Smaller $\eps = 0.01$.}
\end{subfigure} 
\caption{For $\eps = 0.01$ the elements of $D$ behave very similarly for all cases. For a larger value of $\eps$, the elements of $D$ in HermiteGF method show larger magnitude of oscillations which leads to higher $\jmax$.} 
\label{fig:Doscillations}
\end{figure}

\subsection{New HermiteGF cut-off criterion}
In this section, we derive a more holistic criterion for the cut-off in the HermiteGF expansion. We use the Vandermonde formulation of the method since it provides an explicit expression for the elements of all matrices which simplifies the analytical study of the method. We cut the polynomial vector as
\[
H^{G,E,t}(\bx - \bx_0) = 
\begin{pmatrix}\hat{H}^{G, E, t}(\bx - \bx_0) &  H^{G, E, t}_\infty(\bx - \bx_0)\end{pmatrix}
\]
with $\hat{H}^{G, E, t} \in \mathbb{R}^{1\times M}$, where the number 
of basis functions used is larger than the number of collocation points, that is, $M\ge N$.  
Analogously we cut the $N\times \infty$ Vandermonde matrix $W_2$ and the infinite diagonal matrix $D_2$,
\[
W_2 = \begin{pmatrix} \hat{W}_2 & W_\infty\end{pmatrix} \quad\text{and}\quad
D_2 = \begin{pmatrix} \hat{D}_2 & 0\\ 0 & D_\infty\end{pmatrix}
\]
with $\hat{W}_2 \in \mathbb{R}^{N \times (M-N)}$ and $\hat{D}_2 \in \mathbb{R}^{(M-N) \times (M-N)}$. 
We note that the infinite matrix~$W_\infty$ contains the columns of the full Vandermonde matrix $W$ from column $M+1$ onward, while the infinite matrix $D_\infty$ contains the diagonal entries of the full diagonal matrix~$D$ starting from the entry $M+1$. 
We then rewrite the formulation of the method (see \cref{sec:vandermonde}) after the cut-off,
\[
 \hat{\Psi}(\bx)^T = \begin{pmatrix}
                   \mathrm{Id}_{N \times N} & D_1^{-1} W_1^{-1} \hat{W}_2 \hat{D}_2
                  \end{pmatrix} \hat{H}^{G, E, t}(\bx - \bx_0)^T.
\]
We want to make sure that 
\[
 \delta \Psi(\bx) = \Psi(\bx) - \hat{\Psi}(\bx)
\]
is small for all collocation points 
by choosing a sufficiently large but not too large truncation parameter $\jmax$. 
For estimating $\delta\Psi(\bx)$, we need the following lemma. 
\begin{lemma}[Exponential tail]
Consider $\by \in \mathbb{R}^d$ with $y_i \geq 0$ for all $i=1, \ldots, d$ and $\jmax\in\mathbb{N}$. Then,
 \begin{equation}
  \sum_{\vert\lbb\vert \geq \jmax} \frac{\by^{\lbb}}{\lbb!} \leq \sum_{\vert \lbb \vert = \jmax} \exp(\| \by \|_1) \frac{\by^{\lbb}}{\lbb!},
  \label{eq:taylorEstimation}
 \end{equation}
where $\| \by \|_1 = |y_1| + \ldots + |y_d|$.
\label{lemma:taylor}
\end{lemma}
\begin{proof}
Consider the function
\[
 f(\by) = \exp(y_1+\cdots+y_d), \quad \by \geq 0.
\]
We note that the estimated sum coincides with the remainder of the Taylor series for the function $f$ at the point $\mathbf{a}=0$. Then, according to the multivariate Taylor's theorem, there exists $\bxi \in [0, \by]$ such that
\[
\sum_{\vert\lbb\vert \geq \jmax} \frac{\by^{\lbb}}{\lbb!} = 
\sum_{\vert \lbb \vert = \jmax} \partial^{\lbb}f(\bxi)\, \frac{\by^{\lbb}}{\lbb!}.
\]
Noting that $\partial^{\lbb} f(\bxi) = \exp(\mathbf{\bxi}) \leq \exp(\| \by \|_1)$ for $\by\ge0$,  we arrive at \eqref{eq:taylorEstimation}.
\end{proof}
Before proceeding to the estimation of the truncation error $\Vert \delta \Psi(\bx) \Vert_2$, we recall the definition of the vectors 
\[
\mathbf{d_\mathrm{vec}} = \mathrm{diag}(\mathrm{Diag})\quad\text{and}\quad 
\bv_k = (\mathrm{Id} + \mathrm{Diag}^{-1} \mathrm{Rem})\bDelta_k
\] 
for $k=1,\ldots,N$. They will contribute to the upper bound of the following estimate. In the isotropic case, 
they have the particularly simple form
\[
\mathbf{d_\mathrm{vec}} = \gamma^{-1}\epsilon^2(1,\ldots,1)\quad\text{and}\quad 
\bv_k = \bDelta_k\ \text{for all}\ k=1,\ldots,N.
\]

\begin{theorem}[Truncation estimate]
\label{thm:deltaPsiEstimation}
For $k=1,\ldots,N$ we set 
\[
\omega_k = \sum_{i=1}^N (W^{-1}_1)^2_{ki}>0\quad \text{and}\quad  
\by_k = \mathrm{Diag} \, \bv_k\in\mathbb{R}^d,  
\]
where $W_1$ is the upper left $N\times N$ block of the infinite Vandermonde matrix 
$W = (\bv_k^{\lbb})$. For $\jmax \in \mathbb{N}$ we denote
\[
\mathrm{const}_\jmax := 
\left(\sum_{k=1}^N \frac{\omega_k\, \mathbf{k}!\, (t/2)^{|\mathbf{k}|-(\jmax+1)}}{\mathbf{d}_\mathrm{vec}^{2\mathbf{k}}(\jmax+1)!}  \right)
\cdot \left(\sum_{i=1}^N \exp\!\left(\frac{2}{t}\Vert \by_i\Vert^2_2\right) \Vert \by_i \Vert_2^{2(\jmax+1)}\right).
\]
Then, the truncation error $\delta\Psi$ satisfies for all $\bx \in \mathbb{R}^d$,
  \begin{align}
  \Vert \delta \Psi (\bx) \Vert_2^2 &\ \leq\  
  \mathrm{const}_\jmax
  \cdot \left(\haggfe_{\mathrm{lim}}(\bx-\bx_0) - \sum_{\vert \lbb \vert \leq \jmax}\haggf(\bx-\bx_0)^2\right), \label{eq:deltaPsiEstimation}
 \end{align}
 where 
 $\haggfe_{\mathrm{lim}}(\bx-\bx_0)$ can be evaluated via \eqref{eq:HinftyNorm}.
\end{theorem}
\begin{proof}
We start by observing that
\begin{equation*}
 D_1^{-1}W_1^{-1}W_2 D_2 = D_1^{-1}W_1^{-1} \begin{pmatrix}
                                              \hat{W}_2\hat{D}_2 & W_{\infty} D_{\infty}
                                             \end{pmatrix}.
\end{equation*}
Hence, it holds that
\begin{equation*}
 \delta \Psi(\bx)^T = D_1^{-1}W_1^{-1}W_{\infty} D_{\infty} H^{G, E, t}_\infty(\bx - \bx_0)^T,
\end{equation*}
and due to compatibility of Frobenius norm and the 2-norm 
\[
 \Vert \delta \Psi(\bx) \Vert_2^2 \leq \left \Vert D_1^{-1}W_1^{-1}W_{\infty} D_{\infty} \right \Vert_F^2 \cdot \left\Vert H^{G, E, t}_\infty (\bx - \bx_0) \right \Vert_2^2.
\]
We further consider the two norms on the right-hand side separately. We first take a look at the Frobenius norm. Recall that in the RBF-QR method we evaluate the effect of the impact of $D_1^{-1} \ldots D_2$ analytically. We can do the same here:
\[
 D_1^{-1}W_1^{-1}W_{\infty} D_{\infty} = \tilde{D}_\infty \, .* (W_1^{-1}W_{\infty}),
\]
where $.*$ denotes the Hadamard product and 
$\tilde{D}_\infty$ is constructed analogously to \eqref{eq:DtildeAnisotropic}.
We write the Frobenius norm as
\[
\Vert D_1^{-1}W_1^{-1}W_{\infty} D_{\infty} \Vert_F^2 = \sum_{k=1}^N\sum_{\ell> M} \left(\tilde{D}_{k\ell}^2\cdot \left(\sum_{i=1}^N(W_1^{-1})_{ki}W_{i\ell}\right)^2\right)
\]
and estimate with the help of the Cauchy-Schwarz inequality,
\begin{align*}
\tilde{D}_{k\ell}^2 \left(\sum_{i=1}^N(W_1^{-1})_{ki}W_{i\ell}\right)^2
&\le \  
\frac{\omega_k}{D_{kk}^2} \ D_{\ell\ell}^2 \sum_{i=1}^N \bv_i^{2\lbb},
\end{align*}
where $\lbb\in\mathbb{N}^d$ is the $\ell$-th multi-index corresponding to our basis enumeration.
We used the explicit expression of 
$D_{\ell\ell}$ as defined in \eqref{eq:CDanisotropic} and write
\begin{align*}
\frac{\omega_k}{D_{kk}^2} \ D_{\ell\ell}^2 \sum_{i=1}^N \bv_i^{2\lbb}
&= \  
\frac{\omega_k\, \mathbf{k}!\, t^{|\mathbf{k}|}}{\mathbf{d}_\mathrm{vec}^{2\mathbf{k}}\, 2^{|\mathbf{k}|}} \ 
\frac{2^{|\lbb|}}{\lbb!\,t^{|\lbb|}} \sum_{i=1}^N  (\mathrm{Diag}\,\bv_i)^{2\lbb}.
\end{align*}
We denote $\tilde{\by}_i = \begin{pmatrix}\frac{2}{t}(\by_i)_1^2 & \ldots & \frac{2}{t}(\by_i)_d^2\end{pmatrix}$. Then, by \cref{lemma:taylor} we get
\begin{align*}
\Vert D_1^{-1}W_1^{-1}W_{\infty} D_{\infty} \Vert_F^2 
&\leq\ \sum_{k=1}^N  \frac{\omega_k\, \mathbf{k}!\, t^{|\mathbf{k}|}}{\mathbf{d}_\mathrm{vec}^{2\mathbf{k}}\, 2^{|\mathbf{k}|}}  \ \sum_{i=1}^N\sum_{\vert\lbb\vert> \jmax} \frac{\tilde{\by}_i^{\lbb}}{\lbb!}\\
 &\leq \ \sum_{k=1}^N \frac{\omega_k\, \mathbf{k}!\, t^{|\mathbf{k}|}}{\mathbf{d}_\mathrm{vec}^{2\mathbf{k}}\, 2^{|\mathbf{k}|}} \  \sum_{i=1}^N \sum_{\vert \lbb \vert = \jmax+1}\exp\left(\Vert \tilde{\by}_i\Vert_1\right) \frac{\tilde{\by}_i^{\lbb}}{\lbb!}.
 \end{align*}
Using the multinomial theorem and the fact that
 \begin{equation*}
  \Vert \haggfe_\infty(\bx - \bx_0) \Vert_2^2 = \haggfe_{\mathrm{lim}}(\bx - \bx_0) - \sum_{\vert \lbb \vert \leq \jmax} \haggf(\bx - \bx_0)^2.
  \label{eq:tailH2}
 \end{equation*}
 \noindent
we arrive to estimate \eqref{eq:deltaPsiEstimation}.
\end{proof}
  The denominator $\mathbf{d}_\mathrm{vec}^{2\mathbf{k}}$ in  expression \eqref{eq:deltaPsiEstimation} can take extremely small values that can lead to underflow. To avoid this, it can be combined with $
  \Vert \by_i \Vert_2^{2(\jmax + 1)}$. For this, we define the $d$-dimensional index $\mathbf{j}_d = \left(\frac{1}{d},\, ...,\, \frac{1}{d}\right)$ and use the following transformation
\[
 \Vert \by_i \Vert_2 = \mathbf{d}_\mathrm{vec}^{\mathbf{j}_d} \Vert \by_i./( \mathbf{d}_\mathrm{vec}^{\mathbf{j}_d}) \Vert_2 =:   \mathbf{d}_\mathrm{vec}^{\mathbf{j}_d} \Vert\by_i^D \Vert_2,
\]
  where $./$ denotes component-wise division. Pulling out $\mathrm{Diag}^{\mathbf{j}_d}$ the constant of 
  the estimate \eqref{eq:deltaPsiEstimation} can be rewritten as
  \begin{align*}
  \mathrm{const}_{\jmax} = \left(\sum_{k=1}^N \frac{\omega_k \mathbf{k}! \left(\frac{2}{t} \mathbf{d}_\mathrm{vec}^2\right)^{-\mathbf{k} + (\jmax+1)\mathbf{j}_d}}{(\jmax+1)!} \right)\left(\sum_{i=1}^N \exp\!\left(\frac{2}{t}\Vert \by_i\Vert^2_2\right) \Vert \by_i^D \Vert_2^{2(\jmax+1)}\right).
 \end{align*}
\noindent
Note, that in the isotropic case, one can simplify $ \mathbf{d}_\mathrm{vec}^{-2\mathbf{k} + 2(\jmax+1)\mathbf{j}_d} =  (\eps^4/\gamma^2)^{\jmax + 1 - \vert \mathbf{k} \vert}$. We are now ready to formulate our cut-off criterion.
\begin{criterion}
We choose $\jmax$ for the HermiteGF-QR method such that
\begin{equation*}
 \max_{k=1\ldots N}\frac{\Vert \delta \Psi(\bx_k) \Vert_2}{\Vert \hat{\Psi}(\bx_k) \Vert_2} \leq \emph{\texttt{TOL}},
\end{equation*}
where $\{\bx_k\}_{k=1}^N$ are the collocation points.
\end{criterion}
Since we are looking at the relative error, the tolerance \texttt{TOL} need not
be machine precision. The crucial difference to the state-of-the-art criterion \eqref{eq:jmaxOld} is that now the \texttt{TOL} directly controls the accuracy of the stable basis $\Psi$. Depending on the desired accuracy, the tolerance can be adjusted for the specific problem. 
\subsection{Automatic detection of $t$}
One can use the criterion above also for determining the value of the parameter $t$. We scan the whole spectrum of the values of~$t$ and detect the one that yields the minimum amount of basis functions
\[
 \arg \min_{t \in (0,1)} \jmax(t) = t_{\mathrm{auto}}.
\]
Note that very small values of $t$ can cause cancellations and should be excluded (see~\cref{{sec:numerics_iso_N}}). Even though this introduces additional computational cost in the determination of the suitable expansion, it could be profitable for the cases where the basis is used multiple times after having fixed the number $\jmax$ as e.g.\ in a time loop. 

\section{Numerical results}
\label{sec:numericalResults}
We have implemented the HermiteGF interpolation in \texttt{MATLAB}. The code is available for download at \url{https://gitlab.mpcdf.mpg.de/clapp/hermiteGF}. We compare the isotropic HermiteGF-based algorithm with the existing stabilization methods, the Chebyshev-QR method\footnote{Code downloaded from \url{http://www.it.uu.se/research/scientific_computing/software/rbf_qr} on September 10, 2018.} and the Gauss-QR method\footnote{Code downloaded from \url{http://math.iit.edu/~mccomic/gaussqr/} on September 5, 2018.}. We evaluate the influence of different parameters, such as $\eps$, $\gamma$, number of collocation points $N$ on the quality of the interpolation. For the Gauss-QR method, we take the free parameter $\alpha$ to be equal to our value of $\gamma$, i.e., $\alpha_{\mathrm{Gauss-QR}} = \gamma$.
Since there are no stabilization methods available for fully anisotropic interpolation, we test the anisotropic HermiteGF-QR only against the direct algorithm to verify the correctness. To determine the cut-off degree in the HermiteGF method, we use
the HermiteGF cut-off criterion with $\text{\texttt{TOL}} = 10^{-6}$, unless stated otherwise. For this tolerance, the HermiteGF-QR method provides results that match the results from Chebyshev-QR and Gauss-QR. The parameter $t$ is detected automatically. In all tests we evaluate the interpolant at a set of evaluation points $\{\mathbf{z}_k\}_{k=1}^{N_{\rm ev}}$ and look at the average error of the form \cite[$\mathsection$ 5.1, Expr. (5.2)]{fasshauer2012stable}:
\begin{equation*}
 \text{error} = \frac{1}{N_{\rm ev}} \sqrt{\sum_{k=1}^{N_{\rm ev}}
\left(\frac{f(\mathbf{z}_k) - s(\mathbf{z}_k)}{f(\mathbf{z}_k)}\right)^2}.
\end{equation*}

\subsection{2D isotropic interpolation}
\label{sec:hyperIsotropic}
In this section, we take a look at the two-dimensional interpolation with HermiteGF-QR. We take multiples of the identity for both $E$ and $G$. We look at a hyperbolic domain (see \cref{fig:hyperbolicDomain}) defined by the inequality
\begin{equation}
 0.04 \leq (x+1.2)^2- 4y^2 \leq 1 
 \label{eq:hyperbolicDomain}
\end{equation}
with a boundary condition $x^2 + y^2 \leq 1$. The hyperbola of type \eqref{eq:hyperbolicDomain} can then be parameterized as 
\[
(x, y) = r(t) = ( 0.5 c\cosh(t) - 1.2, c\sinh(t)),\, t \in \mathbb{R}, \, c \in [0.2, 1].
\]
\noindent
We run the tests for the following function ($f_4$ from \cite[$\mathsection$ 6]{fornberg2011stable}):
\begin{equation*}
 f_{\mathrm{h}}(x,y) = \sin(x^2 + 2y^2) - \sin(2x^2 + (y - 0.5)^2).
\end{equation*}

\begin{wrapfigure}{l}{0.37\textwidth}
\vspace{-0.5cm}
\hspace{-0.2cm} \includegraphics[scale = 0.37]{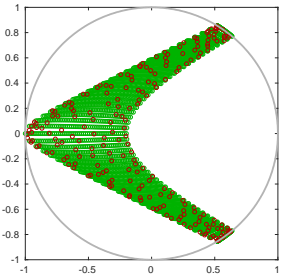}
\caption{Hyperbolic domain. Evaluation grid (green) and $N=210$ clustered Halton node points (red).}
\label{fig:hyperbolicDomain}
\vspace{-0.7cm}
\end{wrapfigure}
\noindent
We investigate the behavior of the performance of the interpolation with respect to the parameters $\gamma$, $\eps$, and number of functions $N$. We use $\gamma = 3.5$, $\epsilon = 0.05$ and optimize $t$ from the set $\mathrm{tvec} = \texttt{linspace(0.1, 0.99, 10)}$, unless stated otherwise.

We sample the collocation points from Halton points that are clustered to the boundary and mapped to the hyperbolic domain.
For all tests, we use $N_{\rm ev} = 53^2$ evaluation points that are sampled similarly to the collocation points, but based on a uniform grid and without clustering. The nodes distribution is depicted in \cref{fig:hyperbolicDomain}. This domain and sampling strategy choice was inspired by \cite[$\mathsection$ 6.1.2] {fornberg2011stable}.

\subsubsection{The number of nodes $N$}\label{sec:numerics_iso_N}

\begin{figure}[t]
 \centering
\begin{subfigure}[t]{0.42\textwidth}
 \includegraphics[scale=0.325]{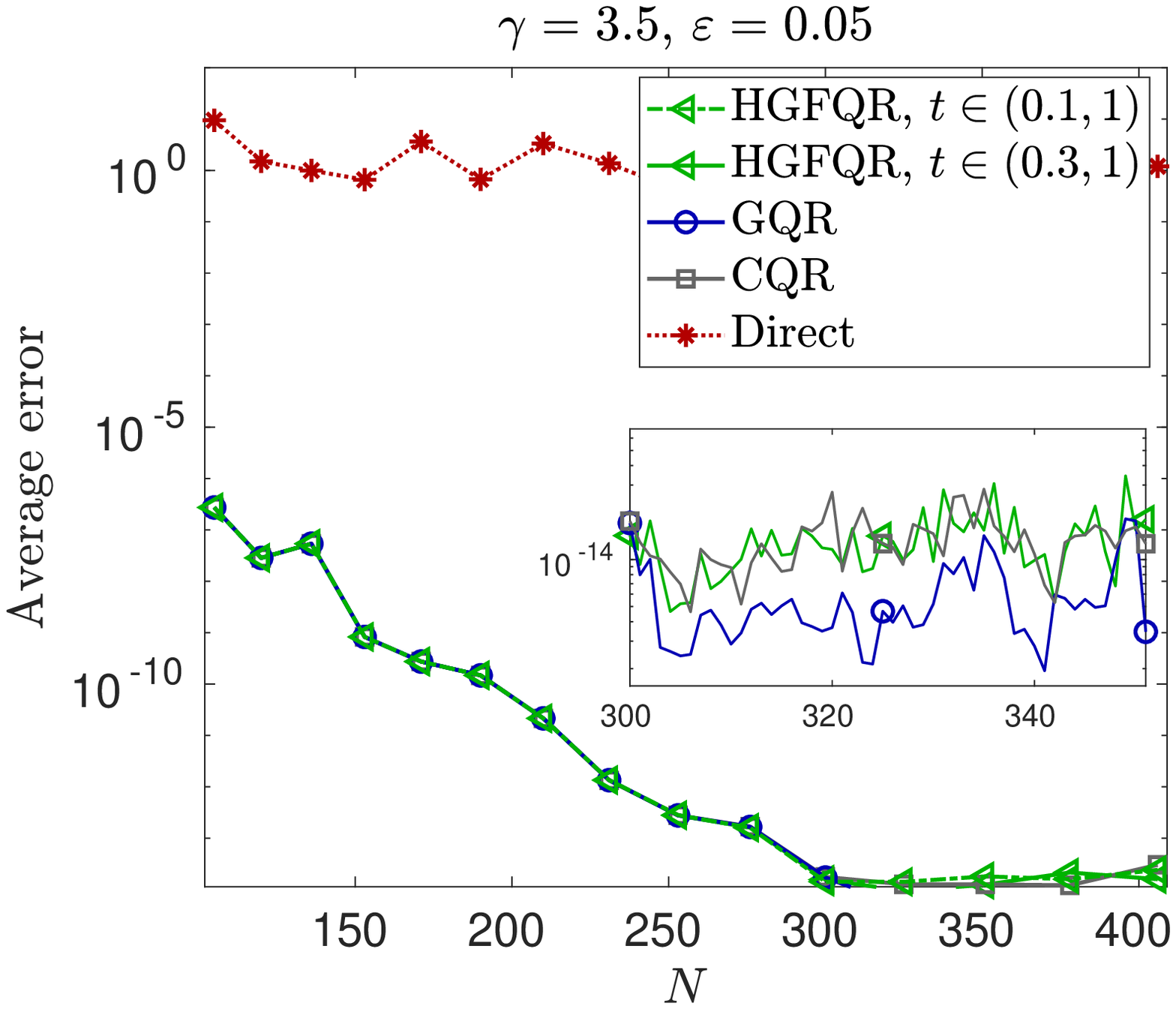}
\caption{Average error.}
\end{subfigure}
\begin{subfigure}[t]{0.45\textwidth}
 \includegraphics[scale=0.325]{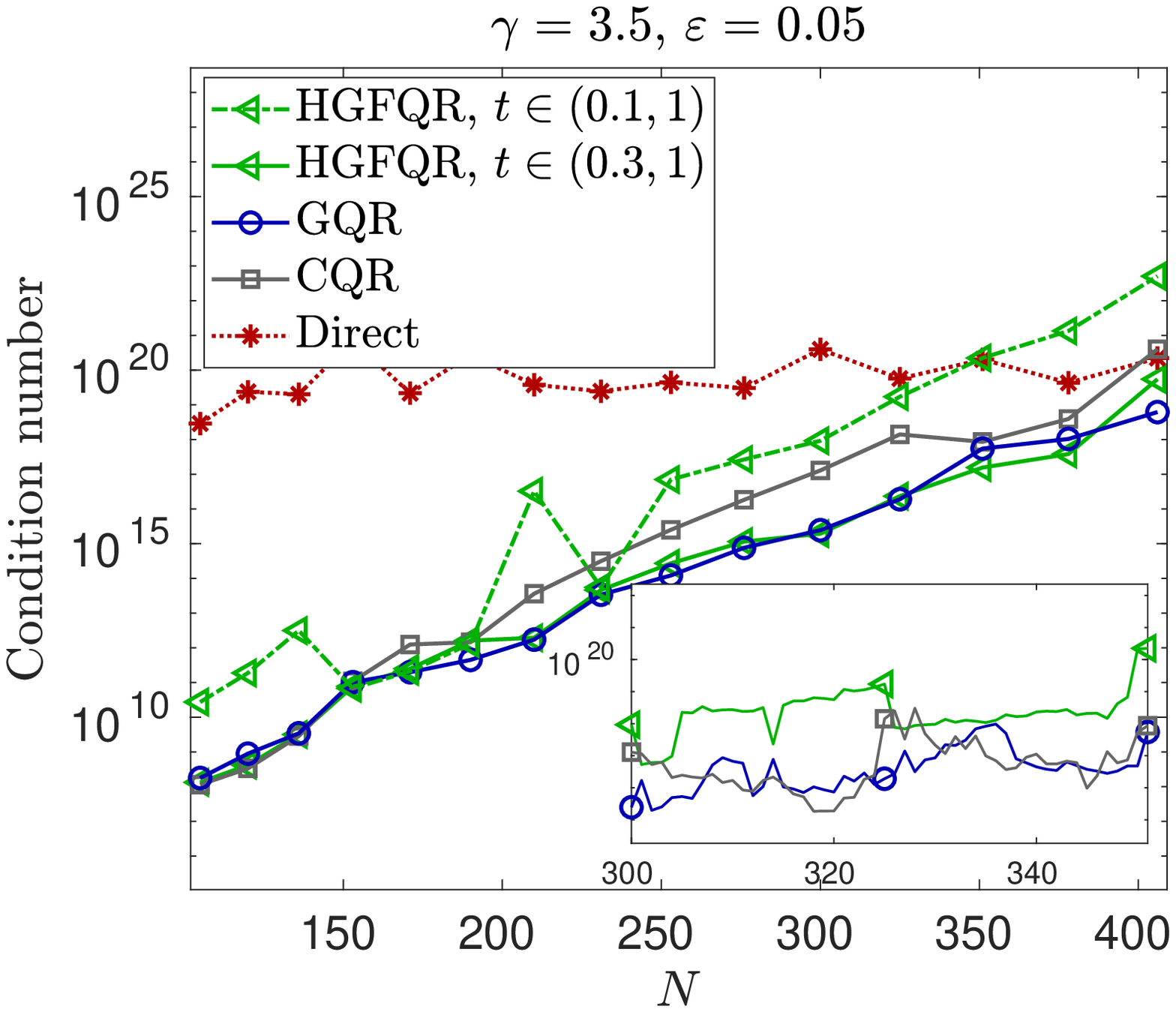}
\caption{Condition number.}
\end{subfigure}
\caption{For the two-dimensional isotropic test case, the interpolation quality is the same for all three stabilization methods. The conditioning is slightly worse for small values of $t$. There is small noise for all methods when the number of radial basis functions $N$ does not correspond to a number of all polynomials of a degree $\leq P$ for a certain $P$.}
\label{fig:varyN}
\end{figure}

Let us first look at the behavior of the method for different numbers of nodes, $N$. We take the values of $N\in[100, 410]$ of the form $\begin{pmatrix}
         P+2 \\
         2                                                                                                                                                                                        \end{pmatrix}$ for some integer $P$, such that there are no same powers of $\eps$ present in both $D_1$ and $D_2$. In \cref{fig:varyN}, we see that the error consistently decays for all the tested methods. Choosing the truncation parameter $t$ in the interval $t\in (0.1, 1)$, the conditioning of the HermiteGF-QR method is slightly worse than for the other methods, since big powers of smaller values of $t$ yield cancellations. Indeed, limiting the range of $t$ to $t\in (0.3, 1)$ brings the conditioning to the level of the other methods. 
         Using all integers in the interval $[100, 410]$ also provided consistent results for all three methods, however, the picture gets noisy. A snippet of that behavior can be seen in the zoomed regions in \cref{fig:varyN}. This can be related to the fact that for the values of $N$ of the form above the limit of the RBF interpolant in the flat limit $\eps \rightarrow 0$ is a \textit{unique} polynomial of degree $P$ \cite[$\mathsection 4$ Theorem 4.1]{larsson2005theoretical} whereas for other values the uniqueness is not guaranteed.
\subsubsection{Sensitivity to $\gamma$}
\label{sec:numericalGamma}
Let us take a look at the influence of the parameter $\gamma$ on the interpolation quality. We see in \cref{fig:gammaN} that for small $N$ the interpolation quality is not sensitive to the value of $\gamma$. However, for larger $N$ the parameter $\gamma$ has to be chosen with care. One can see in the \cref{fig:gammaN_conditioning} that the conditioning is worse for small $\gamma$.
However, one should be careful while increasing $\gamma$ since it also increases the evaluation domain of the 
Hermite polynomials, which take very large values on large domains which can lead to overflow. 
This effect becomes more pronounced as the degree of the Hermite polynomials increases. The optimal balance depends on the particular function and the number of basis functions.
\begin{figure}[h]
 \centering
\begin{subfigure}[t]{0.42\textwidth}
\hspace{-0.2cm}
 \includegraphics[scale=0.325]{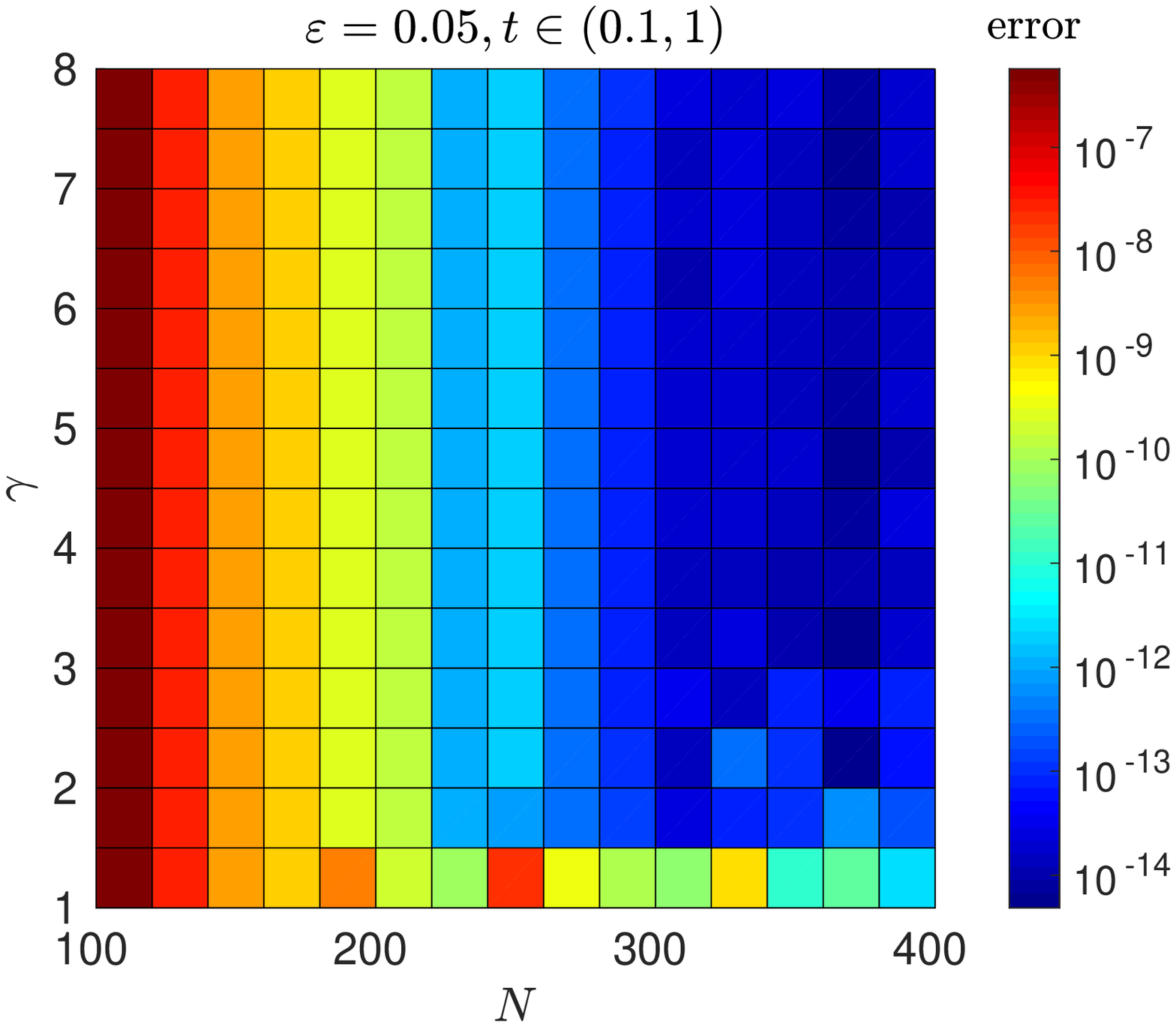}
\caption{Error.}
\end{subfigure}
\begin{subfigure}[t]{0.46\textwidth}
 \includegraphics[scale=0.325]{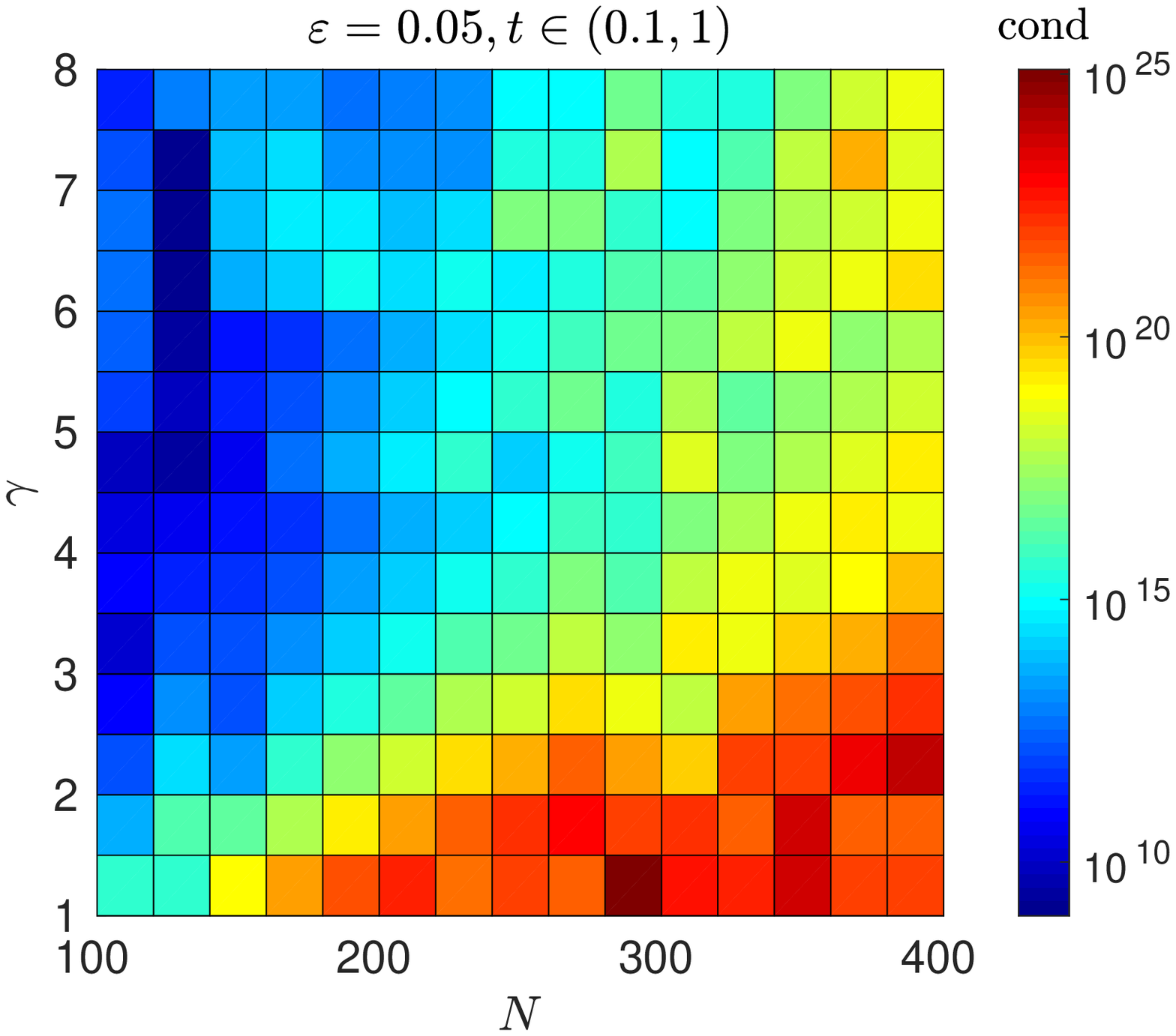}
\caption{Conditioning.}
\label{fig:gammaN_conditioning}
\end{subfigure}
\caption{Average error and condition number for the isotropic two-dimensional test case. 
For small and moderate $N$ the interpolation quality is not sensitive to the value of $\gamma$, whereas for the larger $N$ one should carefully choose the value of $\gamma$.}
\label{fig:gammaN}
\end{figure}
\subsubsection{Cut-off degree $\jmax$}
Next, we look at the influence of the value of \texttt{TOL} on the quality of the interpolation. We compare the error only to the Gauss-QR method since the difference between the Chebyshev-QR and Gauss-QR results is down to machine precision. One can see in  \cref{fig:jmaxVals} that for \texttt{TOL} $=10^{-6}$ the difference HermiteGF-QR and Gauss-QR is also down to machine precision. If we relax the tolerance to $10^{-2}$, the error is still small compared to the magnitude of the interpolation error, while having smaller $\jmax$, which yields an improved computational efficiency. Also, the figure shows a general trend that the expansion decays rather fast for small $\eps$ while an increasing number of basis functions is needed for $\eps$ close to 1.


\subsection{2D anisotropic interpolation}
\begin{figure}[t]
 \begin{subfigure}[t]{0.43\textwidth} 
\includegraphics[scale=0.33]{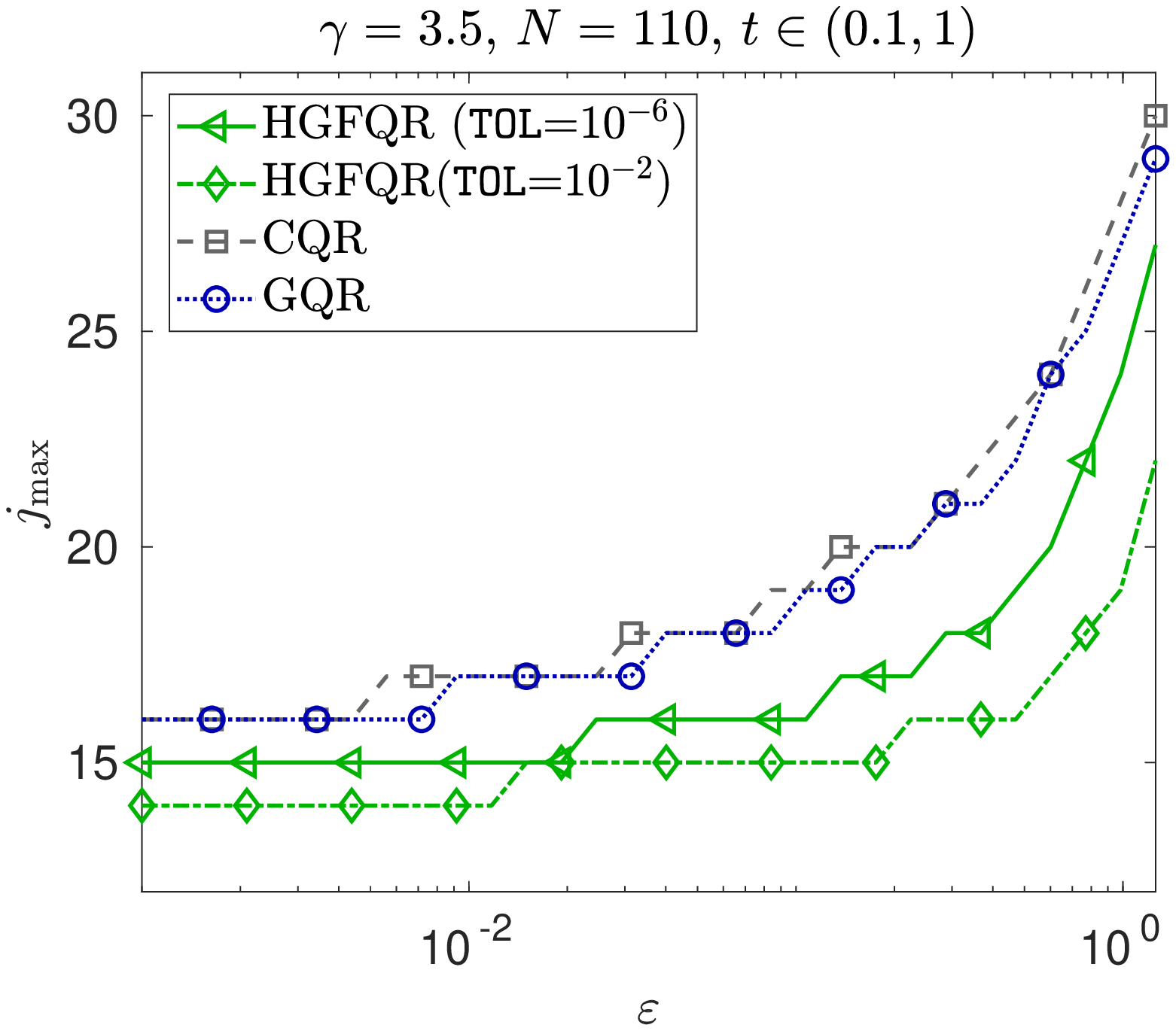} 
 \caption{Dependence of $\jmax$ on $\eps$.}
\end{subfigure} 
\begin{subfigure}[t]{0.465\textwidth} 
\includegraphics[scale=0.33]{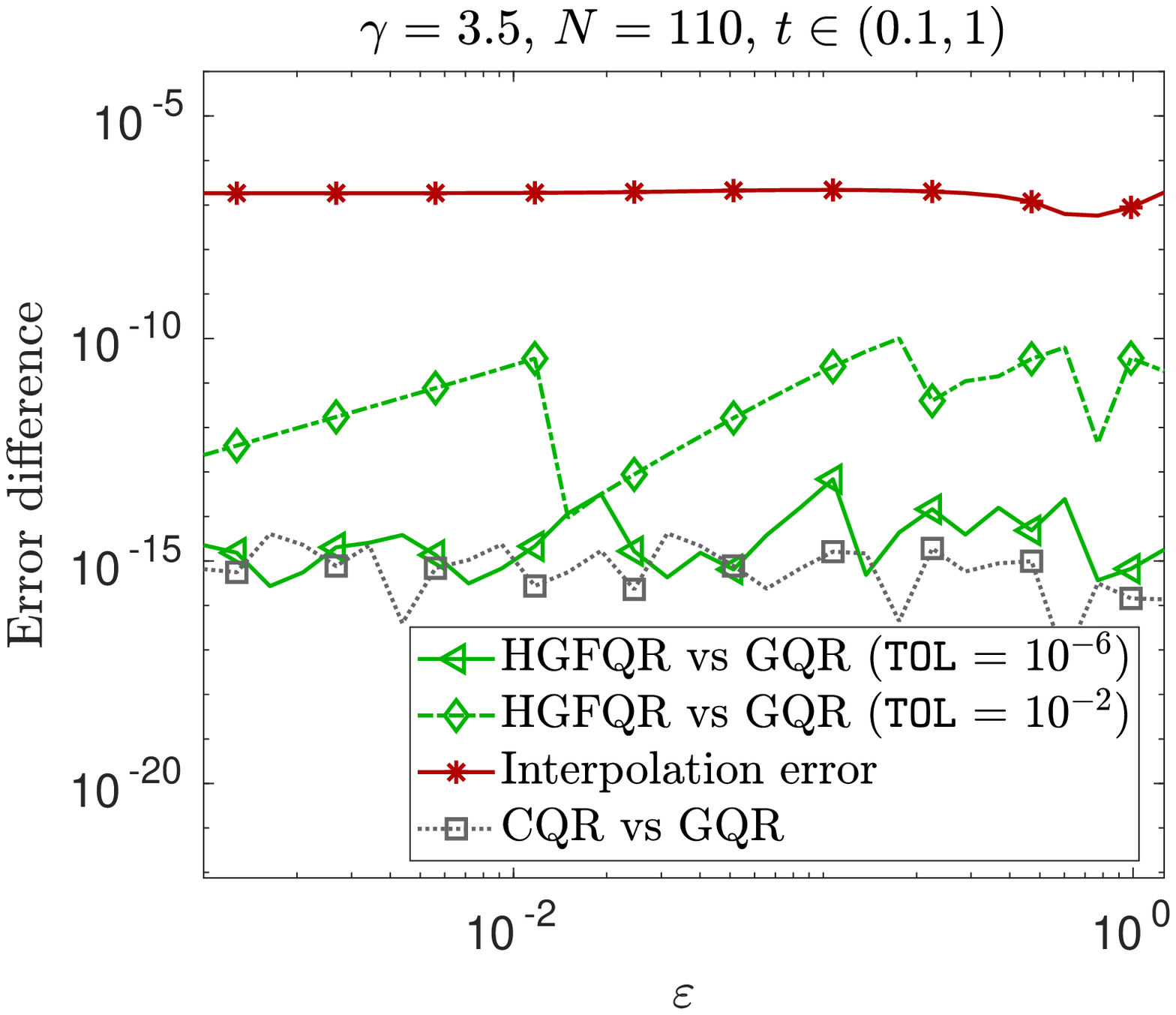}
 \caption{Absolute difference to reference errors.}
\end{subfigure} 
\caption{Optimized truncation value and error differences for the two-dimensional isotropic test case. 
For the coarser \texttt{TOL}$=10^{-2}$ we get fewer basis functions, with truncation error still much below the interpolation error.}
\label{fig:jmaxVals}
\end{figure}
To test the performance of the HermiteGF expansion for anisotropic basis functions, we consider the function:
\begin{equation*}
 f_{\mathrm{a}}(x,y) = \frac{1}{x^2 + xy + y^2} + 2, \quad x,y \in [-1,1].
\end{equation*}
As collocation points, we use Halton points clustered toward the boundaries. For the evaluation grid we use $53 \times 53$ uniformly distributed points.
We choose a non-diagonal matrix $G$. As for the shape matrix $E$, we check whether the off-diagonal elements influence the quality of the results. We fix $E$ and $G$ to be of the following form:
\[
 G = \gamma \begin{pmatrix}
             1 & 0.3 \\
             0.1 & 1.3
            \end{pmatrix}\ \text{with}\ \gamma = 3.5,\quad
 E = \eps \begin{pmatrix}
      1 & p \\
      p & 1
     \end{pmatrix}\ \text{with}\ p \in [0, 0.8].
     \] 
     We restrict $t$ to the interval $\mathrm{tvec} = \texttt{linspace(0.3, 1, 10)}$ in order to improve the stability of the computations. 
Let us take a look on how much the off-diagonal elements of the matrix $E$ influence the error. For our scan, we take 30 logarithmically distributed values of $\eps \in [10^{-3}, 10^{0.1}]$. One can see from \cref{fig:anisotropic_p_N_eps} that properly chosen off-diagonal elements improve the quality of the interpolation. However, for larger $p$ slight instabilities occur which might be explained since $E$ becomes singular for $p\to1$.
\begin{figure}[t]
\begin{subfigure}[t]{0.42\textwidth}
  \centering
  \includegraphics[scale=0.325]{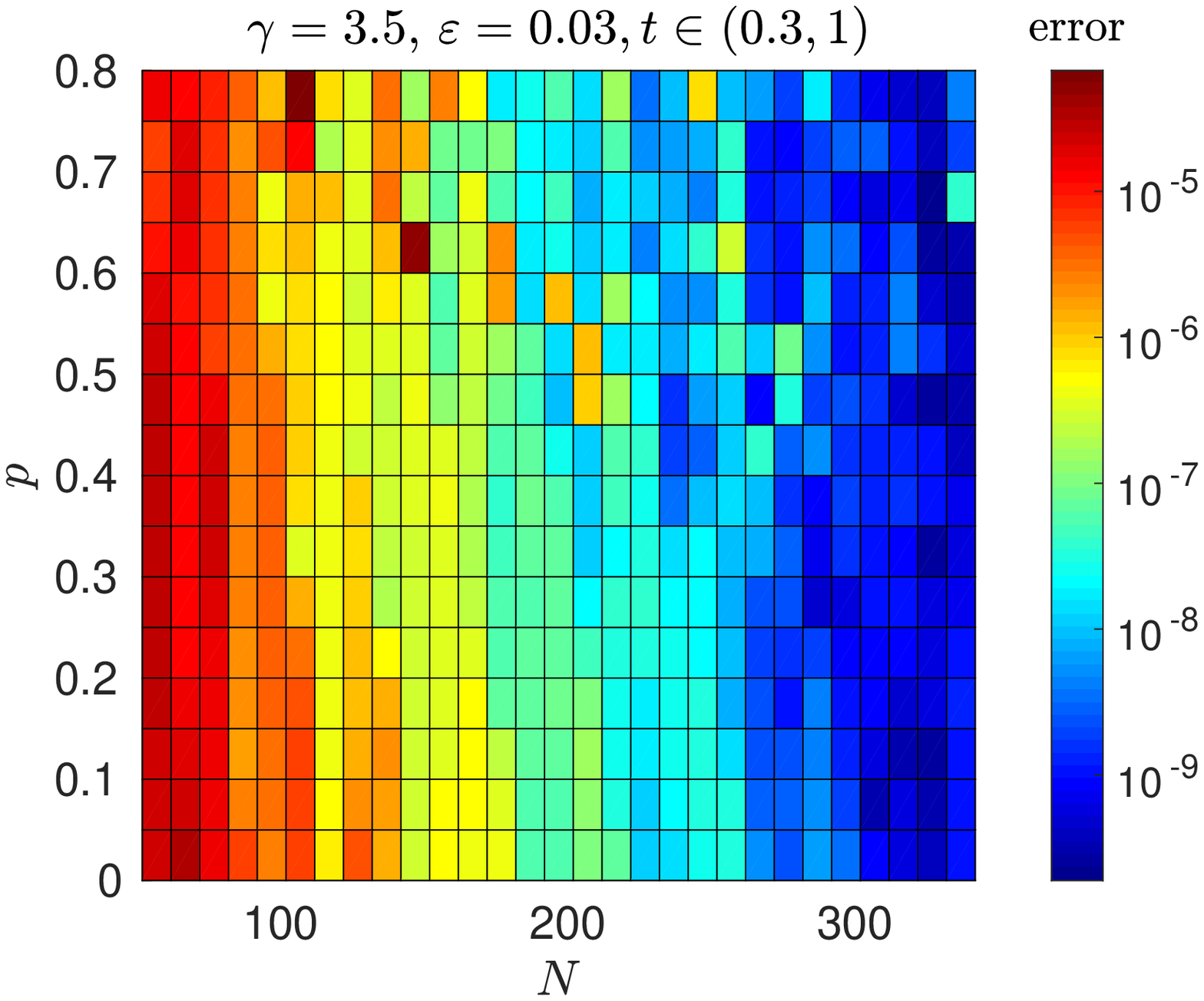}
  \caption{Dependence on $N$.}
 \end{subfigure}
 \begin{subfigure}[t]{0.45\textwidth}
 \centering
  \includegraphics[scale=0.325]{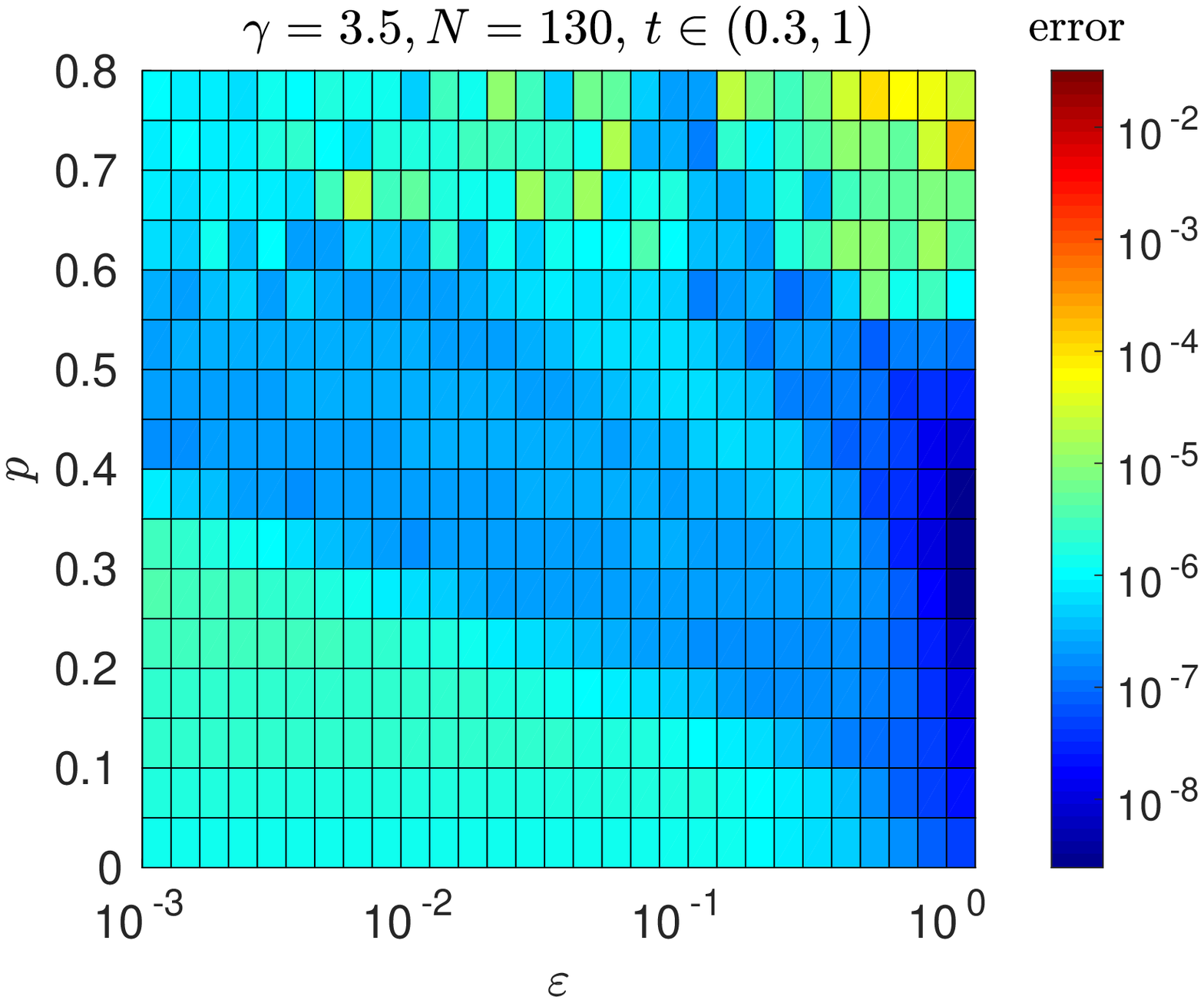}
  \caption{Dependence on $\eps$.}
 \end{subfigure}
\caption{Two-dimensional anisotropic interpolation. For a fixed value of $p$, the error generally decreases with the growth of $N$ as expected. Choosing an anisotropic shape matrix $E$ ($p \neq 0$) often improves the interpolation quality.} 
\label{fig:anisotropic_p_N_eps}
\end{figure}

\subsection{Multivariate interpolation}
In this section, we consider an example of the usage of HermiteGF-QR in higher dimension. For all tests, we use the function
\[
 f(\bx) = \cos(\vert \bx \vert), \quad \bx \in [-1, 1]^d,
\]
where $\vert \bx \vert = \sum_{i=1}^d x_i$. We use Halton collocation points and 1000 Halton points, excluding the ones used for collocation, for the evaluation grid. We fix $G = 5 \mathrm{Id}_d$,  $\eps = $ \texttt{logspace(-3, 0.1, 30)}, 
and we again optimize the parameter $t$ over the set $\mathrm{tvec} = \texttt{linspace(0.3, 1, 10)}$. As before, we choose the tolerance $\texttt{TOL}=10^{-6}$. We first look whether the anisotropic HermiteGF-QR method converges to the results of the direct interpolation as $\eps$ increases. We choose an arbitrary pattern for $E$ in order to verify that the stabilization works for a truly anisotropic interpolation,
\[
 E_{\mathrm{a}} = \eps \begin{pmatrix}
      1 & 0.2 & 0.3 \\
      0.2 & 1 & 0.15 \\
      0.1 & 0.3 & 1
     \end{pmatrix}.
\]
In \cref{fig:anisotropic3d}, we can see that HermiteGF-QR interpolation works stably even for very small values of $\eps$ for different $N$. On the other hand, for larger values of $\eps$ the result matches the direct anisotropic interpolation.

In order to validate the HermiteGF-QR method in higher dimensions against the existing methods, we compare the isotropic HermiteGF-QR method with the Gauss-QR method in 3-5D. For that, we fix 
the shape matrix $E$ and the number of interpolation points $N$ as
\[
E = \mathrm{Id}_d\quad\text{and}\quad N = 4^d,
\]
where $d$ is the dimensionality. We choose the tolerance $\texttt{TOL}=10^{-2}$ since it is enough to meet the overall accuracy of the method. One can see in \cref{fig:isotropic345d} that the HermiteGF-QR method matches the reference Gauss-QR method in 3-5D.

\begin{figure}
\begin{subfigure}[t]{0.42\textwidth}
 \centering
 \includegraphics[scale=0.325]{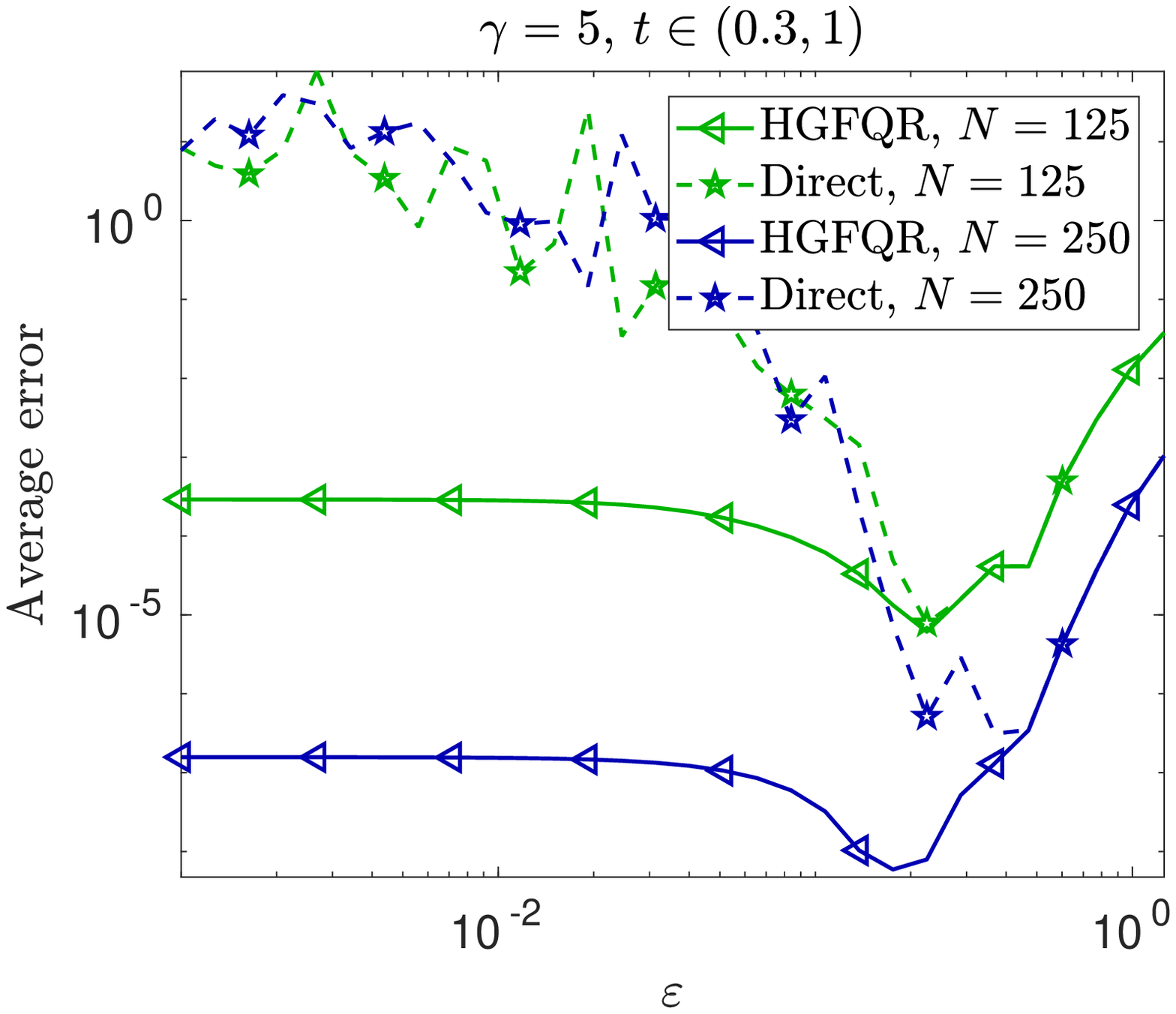}
 \caption{Anisotropic interpolation in 3D.}
 \label{fig:anisotropic3d}
\end{subfigure}
\begin{subfigure}[t]{0.45\textwidth}
 \centering
 \includegraphics[scale=0.325]{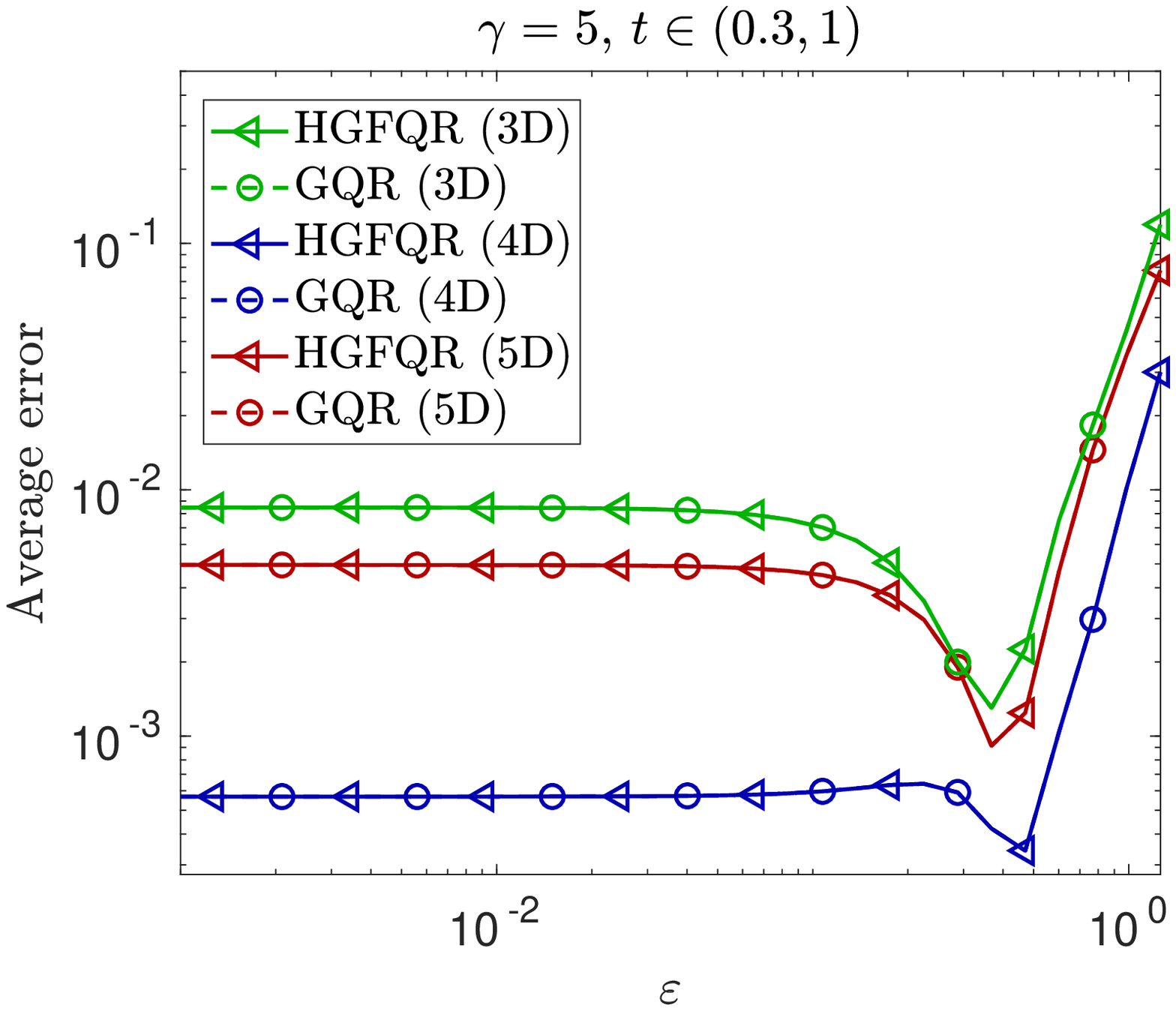}
 \caption{Isotropic interpolation in 3-5D.}
  \label{fig:isotropic345d}
 \end{subfigure}
 \caption{Interpolation in dimensions 3-5D: Both the anisotropic and the isotropic HermiteGF-QR method match the average error of the Gauss-QR method.}
 \label{fig:multidim}
\end{figure}

\section{Conclusion and Outlook}\label{sec:conclusions}
In this paper, we derive a new stabilization algorithm for Gaussian RBF interpolation in the flat limit ($\eps \rightarrow 0$). The main idea of ``isolating'' the ill-conditioning in a special matrix is the same as in the previous approaches \cite{fornberg2007stable,fasshauer2012stable,fornberg2011stable}. 
However, our new algorithmic framework draws from generating functions that naturally extend to
the interpolation with anisotropic Gaussians. 
We introduce several parameters ($\eps$/$E$, $\gamma$/$G$, $t$) for the new HermiteGF basis 
which have a distinct connotation:
$\eps$/$E$ is the original shape parameter of the Gaussian basis, $\gamma$/$G$ stands 
for the size of the evaluation domain of the Hermite polynomials, and the technical basis truncation parameter $t$ can be chosen 
automatically thanks to a novel analytically derived truncation criterion. The interpolation quality is not sensitive to the precise value of $\gamma$/$G$. The generic formulation of the method essentially provides an algorithm in any dimension, and we have reported results for up to five dimensions.

For the cut-off of the HermiteGF expansion, we derive a novel truncation criterion, 
generalizing Mehler's theory for bilinear generating functions. In particular, we analytically estimate the truncation error $\delta \Psi$ in the stable HermiteGF basis $\Psi$. This allows adjusting the number of basis functions based on the desired accuracy in the basis $\Psi$. 

The HermiteGF-QR algorithm has been implemented in \texttt{MATLAB}. For all isotropic test cases, the accuracy of the HermiteGF-QR method is consistent with the ones of 
established stabilization methods (Chebyshev-QR, Gauss-QR). For the anisotropic case, the results matches the RBF-Direct method, where the latter is applicable.

The stability of our algorithm could be further improved by employing special algorithms for the stable inversion of the Vandermonde type matrices. Based on the successful experience of the automatic detection of the truncation parameter $t$, it could be possible to develop an algorithm of choosing an optimal parameter matrix $G$ defining the effective evaluation domain. 
We believe that our anisotropic method could be also of interest in statistical data fitting \cite[$\mathsection$ 17]{fasshauer2015kernel} or for continental size ice sheet simulations, see e.g.\ \cite{cheng2018anisotropic}.
 

\section*{Acknowledgments}
Fruitful discussions with Elisabeth Larsson (Uppsala University) and Michael McCourt (SigOpt) are gratefully acknowledged.

\providecommand{\bysame}{\leavevmode\hbox to3em{\hrulefill}\thinspace}
\providecommand{\MR}{\relax\ifhmode\unskip\space\fi MR }
\providecommand{\MRhref}[2]{%
  \href{http://www.ams.org/mathscinet-getitem?mr=#1}{#2}
}
\providecommand{\href}[2]{#2}

\end{document}